\documentclass[a4paper,11pt, twoside]{article}
\usepackage{amsmath,amsthm}
\usepackage{amssymb,latexsym}
\usepackage{mathrsfs}
\usepackage{enumerate}
\usepackage[colorlinks,citecolor=red]{hyperref}
\usepackage[numbers,sort&compress]{natbib}
\usepackage{indentfirst}
\usepackage{mathtools}
\usepackage{paralist,bbding,pifont}
\newtheorem{theorem}{Theorem}[section]

\newtheorem{lemma}{Lemma}[section]
\newtheorem{proposition}{Proposition}[section]
\newtheorem{definition}{Definition}[section]
\newtheorem{remark}{Remark}[section]

\theoremstyle{definition} \theoremstyle{remark}
\numberwithin{equation}{section}
\allowdisplaybreaks

\begin{document}

\markboth{Y.Z. Yang, Y. Zhou et al}{Time-space fractional Schr\"{o}dinger equation on $\mathbb{R}^{d}$}

\date{}

\baselineskip 0.22in

\title{\bf On the well-posedness of time-space fractional Schr\"{o}dinger equation on $\mathbb{R}^{d}$}

\author{ Yong Zhen Yang$^1$, Yong Zhou$^{1,2}$\\[1.8mm]
\footnotesize {Correspondence: yozhou@must.edu.mo}\\
\footnotesize  {$^{1}$ Faculty of Mathematics and Computational Science, Xiangtan University}\\
\footnotesize  {Hunan 411105, P.R. China}\\[1.5mm]
\footnotesize {$^2$ Macao Centre for Mathematical Sciences, Macau University of Science and Technology}\\
\footnotesize {Macau 999078, P.R. China}\\[1.5mm]
}

\maketitle

\begin{abstract}
This paper considers the well-posedness of a class of time-space fractional Schr\"{o}dinger equations introduced by Naber. In contrast to the classical Schr\"{o}dinger equation, the solution operator here exhibits derivative loss and lacks the structure of a semigroup, which makes the classical Strichartz estimates inapplicable. By using harmonic analysis tools---including the smoothing effect theory of Kenig and Ponce for Korteweg-de Vries equations \cite[\emph{Commun.~Pure Appl.~Math.}]{Kenig}, real interpolation techniques, and the Van der Corput lemma---we establish novel dispersive estimates for the solution operator. These estimates generalize Ponce's regularity results \cite[\emph{J.~Funct.~Anal.}]{Ponce} for oscillatory integrals and enable us to address the derivative loss in the Schr\"{o}dinger kernel.  For the cases $\beta<2$~(in one space dimension) and $\beta>2$~(in higher dimensions), we prove local and global well-posedness in Sobolev and Lorentz-type spaces, respectively. Additionally, we analyze the asymptotic behavior of solutions and demonstrate the existence of self-similar solutions under homogeneous initial data. The results highlight the interplay between fractional derivatives, dispersive properties, and nonlinear dynamics, extending the understanding of nonlocal evolution equations in quantum mechanics and related fields.\\ [2mm]
\end{abstract}

\baselineskip 0.25in

\section{Introduction}
Theoretical and numerical experiments have shown that fractional calculus has significant advantages over traditional integer derivatives in describing the historical dependence of evolutionary processes in complex systems~\cite{Podlubny}. In recent decades, time-space fractional partial differential equations (TSFPDEs) have attracted considerable attention due to their nonlocal effects in both time and space, which make them widely applicable in the real world. In marine settings, the boundary layer surrounding islands can be viewed as an area where discontinuous dynamical behaviors display fractal properties in both spatial and temporal dimensions. This complexity across space and time has motivated our adoption of fractional derivatives in both directions to depict sticky paths on fractal support sets~\cite{Saichev}. In finance, considering the non-Markovian and non-local properties of financial time series, TSFPDEs can accurately characterise the evolution of economic variables over time, which better matches the actual behaviour of returns in financial markets~\cite{Scalas}. More research on fractional partial differential equations can be found in~\cite{A.A.Kilbas,Dong,Y zhou2,Y.Z. Yang,He1,He2}.

The Schr\"{o}dinger equation is one of the cornerstones of quantum mechanics, providing a mathematical framework for understanding the properties of atoms, molecules, and other microscopic particles. The theory of nonlinear Schr\"{o}dinger equations has attracted significant attention from researchers, including topics such as dispersive estimates, Strichartz estimates, and smoothing effects. For references, see \cite{Keel,Nakamura1,Nakamura2,Kenig2}. The nonlinear fractional Schr\"{o}dinger equation, as a generalization of the Schr\"{o}dinger, has important research significance. Laskin \cite{Laskin} considered replacing Brownian paths with L\'{e}vy stable paths in the Feynman path integral technique, leading to the space fractional Schr\"{o}dinger equation
\[
i\partial_{t}w = (-\Delta)^{\frac{\beta}{2}}w + g(w),
\]
where the Markov property of the solution is still preserved. Achar \cite{Achar} used the Feynman path integral technique to derive the time fractional Schr\"{o}dinger equation:
\[
i\partial_{t}^{\alpha}w + \Delta w = g(w).
\]
Naber \cite{Naber} employed the technique of core rotation to change the exponent of \(i\) to \(i^{\alpha}\), providing a physical interpretation, and obtained the following time fractional Schr\"{o}dinger equation:
\[
i^{\alpha}\partial_{t}^{\alpha}w + \Delta w = g(w).
\]
It was later proven that this equation is a classical Schr\"{o}dinger equation with a time-dependent Hamiltonian, making it more suitable for non-stationary quantum problems.

If we fractionalize both the time derivative and the space derivative, we can obtain the following space-time fractional Schr\"{o}dinger equation:
\begin{align}\label{I type}
i^{\mu}\partial^{\alpha}_{t}w - (-\Delta)^{\frac{\beta}{2}}w = g(w), \quad 0 < \alpha < 1, \beta > 0, \quad \mu = \big\{1,\alpha\big\}.
\end{align}
The cases where $\mu = 1$ and $\mu = \alpha$ are the generalised forms of the Schr\"{o}dinger equation studied by Achar \cite{Achar} and Naber \cite{Naber} respectively. Many scientists have done profound research on the E.q.\eqref{I type} when $\mu = 1$. Banquet \cite{Banquet} studied the solvability of a class of space-time Schr\"{o}dinger systems based on E.q.\eqref{I type}. Su \cite{Su1} obtained dispersive estimates for the case $g=0$ by exploiting the asymptotic behaviour of the Mittag-Leffler function $E_{\alpha,1}(-it^{\alpha}|\xi|^{\beta})$. In particular, Su \cite{Su2} skilfully used the asymptotic properties of the H-Fox function to obtain $L^{p}-L^{r}$ estimates for the solution operator of the time-space Schr\"{o}dinger equation, thereby establishing local/global well-posedness in the space $C_{b}\left(0,T; L^{r}(\mathbb{R}^{d})\right) \cap L^{q}\left(0,T;L^{p}(\mathbb{R}^{d})\right)$. This was a significant improvement on previous work. However, research on the $\mu=\alpha$ type time-space Schr\"{o}dinger equation is relatively scarce. Lee \cite{Lee}, combining the asymptotic behavior of the Mittag-Leffler function $E_{\alpha,1}(z)$, studied the Strichartz estimates for E.q.\eqref{I type} when $g=0$, and revealed that the classical Strichartz estimates can be recovered as the ratio $\beta/\alpha \rightarrow 2$. Grande \cite{Grande} used a high-frequency and low-frequency decomposition method, combined with the smoothing effects theory used by Kenig in the study of Korteweg-de Vries equations, to investigate the local well-posedness and ill-posedness of E.q.\eqref{I type} for $0 < \beta < 2$ and $d=1$.

In this paper, we consider the following time-space fractional Schr\"{o}dinger equation:
\begin{align}\label{Eq:F-T-S-E}
    \begin{cases}
        i^{\alpha}\partial_{t}^{\alpha}w(t,x)-(-\Delta)^{\frac{\beta}{2}}w(t,x)+g(w(t,x))=0 & \text{in } (0,\infty) \times \mathbb{R}^{d}, \\
        w(0,x)=w_{0}(x) & \text{in } \mathbb{R}^{d},
    \end{cases}
\end{align}
where \( 0 < \alpha < 1 \), \( \beta > 0 \). The operator \( \partial_{t}^{\alpha} \) denotes the \(\alpha\)-Caputo derivative, and \( (-\Delta)^{\frac{\beta}{2}} \) (\( \beta > 0 \)) represents the fractional Laplace operator, which is defined in the sense of Fourier multipliers as $(-\Delta)^{\frac{\beta}{2}}f=\mathcal{F}^{-1}\big(|\xi|^{\beta}\widehat{f}(\xi)\big)$. The function \( g(w) \) represents the nonlinear term.

The study of the Schr\"{o}dinger equation with $\mu=\alpha$ is more challenging than the $\mu=1$ type, as the techniques used for the $\mu=1$ type are not applicable to the $\mu=\alpha$ type. The main reason is that when $z = -it^{\alpha}|x|^{\beta}$ approaches infinity along the ray $|\arg{z}| = \pi/2$, the solution operator of the equation of $\mu=1$ type exhibits polynomial or exponential decay. However, when $z = (-it)^{\alpha}|x|^{\beta}$ approaches infinity along the ray $|\arg{z}| = \alpha\pi/2$, the solution operator of the $\mu=\alpha$ type equation shows oscillatory behavior, which prevents the establishment of the $L^{p}-L^{r}$ estimates that can be obtained for the $\mu=1$ type equation (see \cite{Su2}). On the other hand, combining the asymptotic behavior of the Mittag-Leffler function, as $z \rightarrow \infty$,
\begin{align*}
E_{\alpha,\beta}(z) = \frac{1}{\alpha} \exp(z^{\frac{1}{\alpha}}) z^{\frac{1-\beta}{\alpha}} - \sum_{j=1}^{m} \frac{z^{-j}}{\Gamma(\beta-\alpha j)} + O(|z|^{-m-1}), \quad |\arg{z}| \leq \alpha\pi/2,
\end{align*}
and
\begin{align*}
E_{\alpha,\beta}(z) =- \sum_{j=1}^{m} \frac{z^{-j}}{\Gamma(\beta-\alpha j)} + O(|z|^{-m-1}), \quad \frac{\alpha\pi}{2}<|\arg{z}| \leq \pi.
\end{align*}
 In contrast to the \(\mu=1\) type space-time fractional Schr\"{o}dinger equation, the solution operator for the \(\mu=\alpha\) type space-time fractional Schr\"{o}dinger equation involves a derivative loss in the Schr\"{o}inger kernel $D^{\delta-\beta}\exp(-itD^{\beta})$, which does not belong to any H\"{o}rmander multiplier class. This is fundamentally different from the fractional heat kernel, making the H\"{o}rmander multiplier theorem used in \cite{Miao, Y.Z. Yang} inapplicable to the solution operators. This presents significant difficulties in establishing dispersive estimates for the solution operator. Inspired by \cite{Kenig, Grande, Lee, Ponce, Guozihua}, this paper bridges these gaps through a unified harmonic analysis framework. By synthesizing four key methodologies---
\begin{itemize}
    \item \textbf{Fractional heat-Schr\"{o}dinger kenels} for solution operator equivalent representation,
    \item \textbf{Smoothing effects} for derivative-loss operators via Kenig-Ponce theory \cite{Kenig},
    \item \textbf{Oscillatory integral estimates} generalizing Ponce's regularity results \cite{Ponce},
    \item \textbf{Dispersive-analytic techniques} for fractional heat-Schr\"{o}dinger kernels,
\end{itemize}
---we establish novel well-posedness theories for the nonlinear time-space fractional Schr\"{o}dinger equation \eqref{Eq:F-T-S-E} in both $\beta < 2$ and $\beta > 2$ regimes:
\begin{enumerate}
    \item \textbf{For $\beta > 2$ ($d \geq 1$)}: Global existence in Lorentz-type spaces $X_{p_0}^\kappa$ under small initial data, leveraging generalized dispersive estimates (Lemma \ref{Mittag-Leffer operator estimate}) and contraction mapping.
    \item \textbf{For $\beta < 2$ ($d = 1$)}: Local well-posedness in $C(0,T; H^s(\mathbb{R}))$ via frequency-localized smoothing effects (Lemmas \ref{smoothing effects 1}--\ref{interpolation estimate}), addressing derivative loss through Bourgain-type spaces.
    \item \textbf{Asymptotics and self-similarity}: Decay properties (Theorem \ref{asymptotic behavior}) and scale-invariant solutions (Theorem \ref{self-similar solution}) under homogeneous initial data, exploiting the kernel's fractional scaling.
\end{enumerate}
 These results significantly extend the mathematical understanding of nonlocal dispersive equations, particularly for equations lacking standard semigroup structures. The methodology developed here may apply to other fractional models with memory-driven dynamics.

This paper is organized as follows. In Section 2, we introduce some of the notations required for this paper, including fractional derivatives, Lorentz spaces, fractional Leibniz rule, and Van der Corput lemma etc. In Section 3, we prove some dispersive estimates for the solution operator of the E.q.\eqref{Eq:F-T-S-E}, which play a critical role in establishing the main results of this paper. In Section 4, we establish the global/local well-posedness of the E.q.\eqref{Eq:F-T-S-E} and analyze the asymptotic behavior and self-similar properties.

\section{Notations}
Some of the symbols, definitions, and lemmas used in this article are introduced in this section.

In this paper, $C$ denotes a general constant which may vary between different lines. The notation $f \lesssim g$ indicates that there exists a positive constant $C$ such that $f \leq C g$. Let $X$ be a Banach space, \( I \in \mathbb{R}^{+} \)~and denote $X^{\star}$ that its dual space. We denote \( C\big(I;X\big) \) as the space of \( X \)-valued continuous functions on \( I \), endowed with the norm $\big\|w\big\|_{C} = \sup_{t \in I} \big\|w(t)\big\|_{X}$. We denote \( L^{p}(I;X) \) as the space of \( L^{p} \)-integrable functions on \( I \), endowed with the norm \( \big\|w\big\|_{L^{p}(I;X)} \). Let \(\mathcal{S}(\mathbb{R}^{d})\) denotes the class of Schwartz functions, and \(\mathcal{S'}(\mathbb{R}^{d})\) stands for tempered distributions. Moreover, we denote $\dot{\mathcal{S}}(\mathbb{R}^{d})$ that is
$$
\dot{\mathcal{S}}(\mathbb{R}^{d})=\big\{f\in\mathcal{S}:\text{ for all multiindex }\gamma\in \mathbb{N}^{d}, \partial^{\gamma}f(0)=0\big\}.
$$
We denote the convolution operations with respect to the time variable \( t \) and the space variable \( x \) by \( * \) and \( \star \), respectively. The symbols \(\mathcal{L}\) and \(\mathcal{L}^{-1}\) respectively represent the Laplace transform and the inverse Laplace transform. Let \(\mathcal{F}\) and \(\mathcal{F}^{-1}\) respectively denote its Fourier transform and inverse Fourier transform, defined by
\[
\big(\mathcal{F}f\big)(\xi) = \widehat{f}(\xi)=\int_{\mathbb{R}^{d}}f(x)e^{-ix\cdot\xi}\,dx,
\]
and
\[
\big(\mathcal{F}^{-1}f\big)(x) =(2\pi)^{\frac{d}{2}} \check{f}(x) =\int_{\mathbb{R}^{d}}f(\xi)e^{ix\cdot\xi}\,d\xi.
\]
And by using dual method, the Fourier transform and Fourier inverse transform can extend to the tempered distribution, that is $\langle\mathcal{F}f,\phi\rangle=\langle f,\mathcal{F}\phi\rangle$,$\langle\mathcal{F}^{-1}f,\phi\rangle=\langle f,\mathcal{F}^{-1}\phi\rangle$ for $f\in\mathcal{S}'$ and any $\phi\in\mathcal{S}$.
Moreover, when $f$ is a measurable function on $\mathbb{R}^{d}$ with a bound on its polynomial growth at infinity, we describe the operator $f(\nabla)$ through the formula $f(\nabla)a \equiv \mathcal{F}^{-1}(f\mathcal{F}a)$ with $\nabla=\frac{\partial}{\partial_{x}}$. For convenience, we denote $D=|\nabla|=\nabla\circ \mathbb{H}$, where $\mathbb{H}$ is the Hirbert transform.

Thus, $e^{-t(-\Delta)^{\frac{\beta}{2}}}$ is a pseudo-differential operator quantized from the symbol $e^{-t|\xi|^{\beta}}$, and it can be defined as a convolution operator represented by the fractional heat kernel $S_{t}(x)$,
\begin{equation*}
e^{-t(-\Delta)^{\frac{\beta}{2}}}g(x) = \mathcal{F}^{-1}\left(e^{-t|\xi|^{\beta}}\mathcal{F}g\right)(x) = \mathcal{F}^{-1}(e^{-t|\xi|^{\beta}}) \star g(x) = S_{t} \star g(x),
\end{equation*}
where
\begin{equation*}
S_{t}(x) = \mathcal{F}^{-1}(e^{-t|\xi|^{\beta}})(x) = t^{-\frac{d}{\beta}}S\left(t^{-\frac{1}{\beta}}x\right), \; S(x) = (2\pi)^{-\frac{d}{2}}\int_{\mathbb{R}^{d}}e^{ix\xi}e^{-|\xi|^{\beta}}d\xi.
\end{equation*}
Similarly, for any $\theta \geq 0$, $(-\Delta)^{\frac{\theta}{2}}e^{-t(-\Delta)^{\frac{\beta}{2}}}$ is a pseudo-differential operator with the symbol $|\xi|^{\theta}e^{-t|\xi|^{\beta}}$, and can be defined as
\begin{equation*}
(-\Delta)^{\frac{\theta}{2}}e^{-t(-\Delta)^{\frac{\beta}{2}}}g(x) = \mathcal{F}^{-1}\left(|\xi|^{\theta}e^{-t|\xi|^{\beta}}\right) \star g(x) = S_{t,\theta} \star g(x),
\end{equation*}
where
\begin{equation*}
S_{t,\theta}(x) = t^{-\frac{\theta}{\beta}}t^{-\frac{d}{\beta}}S_{\theta}\left(t^{-\frac{1}{\beta}}x\right), \; S_{\theta}(x) = (2\pi)^{-\frac{d}{2}}\int_{\mathbb{R}^{d}}e^{ix\xi}|\xi|^{\theta}e^{-|\xi|^{\beta}}d\xi.
\end{equation*}
For further details on $e^{-t(-\Delta)^{\frac{\beta}{2}}}$ and $(-\Delta)^{\frac{\theta}{2}}e^{-t(-\Delta)^{\frac{\beta}{2}}}$, refer to \cite{Miao}.

For $1\leq p,q\leq\infty$, $g:[0,\infty)\rightarrow \mathbb{R}^{d}$, we denote that
$$
\big\|g\big\|_{L^{p}_{x}L^{q}_{T}}=\left(\int_{\mathbb{R}^{d}}\left(\int_{0}^{T}\big|g(t,x)\big|^{q}dt\right)^{\frac{p}{q}}dx
\right)^{\frac{1}{p}}
$$
and
$$
\big\|g\big\|_{L^{p}_{T}L^{q}_{x}}=\left(\int_{0}^{T}\left(\int_{\mathbb{R}^{d}}\big|g(t,x)\big|^{q}dx\right)^{\frac{p}{q}}dt
\right)^{\frac{1}{p}}
$$
and with $T=t$ to represent the case when $[0,T]=\mathbb{R}^{+}$.

We define the space $BMO_{x}L^{2}_{T}$ that is
$$
BMO_{x}L^{2}_{T}=\left\{f\in L_{loc}^{1}L^{2}_{T}/\frac{1}{|Q|}\int_{Q}\|f-f_{Q}\|_{2}dx<\infty, f_{Q}=\frac{1}{|Q|}\int_{Q}fdx,\text{ for any open ball }Q\in\mathbb{R}^{d}\right\}.
$$

From Kenig \cite{Kenig}, for any $\omega\in \mathbb{R}$, there exists the constant $C_{\omega}$ such that
\begin{align}\label{BMO estimate}
\big\|D^{i\omega}f\big\|_{BMO_{x}L^{2}_{t}}\leq C_{\omega}\big\|f\big\|_{L^{\infty}_{x}L^{2}_{t}},\quad f\in L^{\infty}_{x}L^{2}_{t}.
\end{align}

For \( s \in \mathbb{R} \), $1\leq p\leq\infty$, the homogeneous and inhomogeneous Sobolev spaces \( \dot{H}^{s}_{p}(\mathbb{R}^{d}) \) and \( H^{s}_{p}(\mathbb{R}^{d}) \) can be represented as

\begin{align*}
\dot{H}^{s}_{p}(\mathbb{R}^{d})=\left\{f\in\mathcal{S}'(\mathbb{R}^{d}):\mathcal{F}^{-1}\left(|\xi|^{s}\widehat{f}\right)\in L^{p}(\mathbb{R}^{d})\right\}
\end{align*}
and
\begin{align*}
H^{s}_{p}(\mathbb{R}^{d})=\left\{f\in\mathcal{S}'(\mathbb{R}^{d}):\mathcal{F}^{-1}\left(\big(1+|\xi|^{2}\big)^{\frac{s}{2}}
\widehat{f}\right)\in L^{p}(\mathbb{R}^{d})\right\}
\end{align*}
with the norms
\begin{align*}
\big\|f\big\|_{\dot{H}^{s}_{p}}=\big\|(-\Delta)^{\frac{s}{2}}f\big\|_{L^{p}(\mathbb{R}^{d})}
=\big\|D^{s}f\big\|_{L^{p}(\mathbb{R}^{d})}, \quad \big\|f\big\|_{H^{s}_{p}}=\big\|(I-\Delta)^{\frac{s}{2}}f\big\|_{L^{p}(\mathbb{R}^{d})}
=\big\|\langle D\rangle^{s} f\big\|_{L^{p}(\mathbb{R}^{d})}.
\end{align*}
In particular, for \( s > 0 \), we have
\[
H^{s}_{p}(\mathbb{R}^{d}) = L^{p}(\mathbb{R}^{d}) \cap \dot{H}^{s}_{p}(\mathbb{R}^{d}),
\]
and for any $\varepsilon>0$, $s\in\mathbb{R}$, we have the embedding
$$
H^{s+\varepsilon}_{p}(\mathbb{R}^{d})\hookrightarrow H^{s}_{p}(\mathbb{R}^{d}).
$$
Sometimes, we denote the Sobolev spaces \(H^{s}_{p}(\mathbb{R}^{d})\) and \(\dot{H}^{s}_{p}(\mathbb{R}^{d})\) by \(L^{p}_{-s}(\mathbb{R}^{d})\) and \(\dot{L}^{p}_{-s}(\mathbb{R}^{d})\), respectively.

Next we recall the Lorentz space $L^{p,q}(\mathbb{R}^{d})$, $1\leq p,q\leq\infty$, that is defined
$$
L^{p,q}(\mathbb{R}^{d})=\left\{g:\mathbb{R}^{d}\rightarrow\mathbb{C}:\left\|g\right\|_{L^{p,q}}<\infty\right\},
$$
where
\begin{align*}
\left\|g\right\|_{L^{p,q}}=
\begin{cases}
\left(\int_{0}^{\infty}\left(t^{\frac{1}{p}}g^{*}(t)\right)^{q}\frac{dt}{t}\right)^{\frac{1}{q}}\quad q<\infty,\\
\sup_{t>0}t^{\frac{1}{p}}g^{*}(t)\quad q=\infty.
\end{cases}
\end{align*}
and $g^{*}$ is the decreasing rearrangement function of $g$, can refer to \cite{Grafakos}. In particular, we have that $L^{p,p}(\mathbb{R}^{d})=L^{p}(\mathbb{R}^{d})$.
\begin{definition}
The Riemann-Liouville fractional integral with $0<\alpha<1$ is defined as
\[
{}_{0}J_{t}^{\alpha}\!f(t,x)=\frac{1}{\Gamma(\alpha)}\int_{0}^{t}(t-s)^{\alpha-1}f(s,x)ds=g_{\alpha}(t)*f(t,x),\quad f\in L^{1}\big(0,\infty;\mathcal{S}(\mathbb{R}^{d})\big),
\]
where $g_{\alpha}(t)=\frac{t^{\alpha-1}}{\Gamma(\alpha)}$.
\end{definition}
\begin{definition}
The Caputo fractional derivative with $0<\alpha<1$ can be  defined as
\begin{equation*}
\partial^{\alpha}_{t}f(t,x)=\frac{d}{dt}\left(g_{1-\alpha}(t)*(f(t,x)-f(0,x))\right),\quad f\in L^{1}\in L^{1}\big(0,\infty;\mathcal{S}(\mathbb{R}^{d})\big).
\end{equation*}
\end{definition}

Next, we introduce the Mittag-Leffler function $E_{\alpha,\beta}(z)$, which is defined as follow:
\begin{equation*}
E_{\alpha,\beta}(\rho) = \sum_{n=0}^{\infty} \frac{\rho^{n}}{\Gamma(\alpha n + \beta)}, \qquad \text{for }\alpha, \beta > 0, \text{ and }\rho \in \mathbb{C},
\end{equation*}

The Mittag-Leffler function has the following properities, can refer to \cite{A.A.Kilbas,Gorenflo}.
\begin{proposition}\label{prop:Mittag-Leffer}
For the Mittag-Leffler function \( E_{\alpha,\beta}(z) \), the following properties hold:

\begin{enumerate}[\rm(i)]
    \item If \( 0 < \alpha < 1\), \( \beta < 1 + \alpha \), \( z \neq 0 \), and \( |\arg(z)|<\alpha\pi\), then
    \begin{equation}
    E_{\alpha,\beta}(z) = \frac{1}{\alpha} z^{\frac{1-\beta}{\alpha}} \exp\left(z^{\frac{1}{\alpha}}\right)
    + \int_{0}^{\infty} \frac{1}{\pi \alpha} r^{\frac{1-\beta}{\alpha}} \exp\left(-r^{\frac{1}{\alpha}}\right)
    \frac{r \sin(\pi(1-\beta)) - z \sin(\pi(1-\beta + \alpha))}{r^2 - 2rz \cos(\pi \alpha) + z^2} \, dr.
    \end{equation}
   \item The following Laplace transform equation holds:
    \begin{equation}
    \int_{0}^{\infty} e^{-st} t^{\beta-1} E_{\alpha,\beta}(\pm a t^{\alpha}) \, dt = \frac{s^{\alpha-\beta}}{s^{\alpha} \mp a} \quad \text{for} \quad \Re(s) > 0, \, a \in \mathbb{C}, \, |s^{-\alpha} a| < 1.
    \end{equation}
\end{enumerate}
\end{proposition}

\begin{remark}\label{change variable equaton of Mittag-Leffer function}
If we take \( z = (-it)^{\alpha} |\xi|^{\beta} \), for \( E_{\alpha,1}((-it)^{\alpha} |\xi|^{\beta}) \) and \( E_{\alpha,\alpha}((-it)^{\alpha} |\xi|^{\beta}) \), we respectively have the following results by using the variable transform \( r = \tilde{r}^{\alpha} t^{\alpha} |\xi|^{\beta} \):
\begin{align*}
    &E_{\alpha,1}((-it)^{\alpha} |\xi|^{\beta}) \\
    &= \frac{1}{\alpha} e^{-it |\xi|^{\frac{\beta}{\alpha}}}
    - \frac{(-it)^{\alpha} |\xi|^{\beta} \sin(\alpha \pi)}{\pi \alpha} \int_{0}^{\infty}
    \frac{e^{-\tilde{r} t |\xi|^{\frac{\beta}{\alpha}}} \alpha \tilde{r}^{\alpha-1} t^{\alpha} |\xi|^{\beta}}
    {\tilde{r}^{2\alpha} t^{2\alpha} |\xi|^{2\beta} - 2 (-i)^{\alpha} \tilde{r}^{\alpha} t^{2\alpha} |\xi|^{2\beta} \cos(\alpha \pi)
    + (-i)^{2\alpha} t^{2\alpha} |\xi|^{2\beta}} \, d\tilde{r} \\
    &= \frac{1}{\alpha} e^{-it |\xi|^{\frac{\beta}{\alpha}}}
    - \frac{\sin(\alpha \pi)}{\pi} \int_{0}^{\infty}
    \frac{e^{-rt |\xi|^{\frac{\beta}{\alpha}}} r^{\alpha-1}}
    {i^{\alpha} r^{2\alpha} - 2 r^{\alpha} \cos(\alpha \pi) + (-i)^{\alpha}} \, dr, \\
    \\
    &E_{\alpha,\alpha}((-it)^{\alpha} |\xi|^{\beta}) \\
    &= \frac{1}{\alpha} (-i)^{1-\alpha} t^{1-\alpha} |\xi|^{\frac{\beta}{\alpha}(1-\alpha)} e^{-it |\xi|^{\frac{\beta}{\alpha}}} \\
    &\quad - \frac{\sin(\pi(1-\alpha))}{\alpha \pi} \int_{0}^{\infty}
    \frac{\tilde{r} t |\xi|^{\frac{\beta}{\alpha}} e^{-\tilde{r} t |\xi|^{\frac{\beta}{\alpha}}} \alpha \tilde{r}^{\alpha-1} t^{\alpha} |\xi|^{\beta}}
    {\tilde{r}^{2\alpha} t^{2\alpha} |\xi|^{2\beta} - 2 \tilde{r}^{\alpha} \cos(\alpha \pi) (-i)^{\alpha} t^{2\alpha} |\xi|^{2\beta}
    + (-i)^{2\alpha} t^{2\alpha} |\xi|^{2\beta}} \, d\tilde{r} \\
    &= \frac{1}{\alpha} (-i)^{1-\alpha} t^{1-\alpha} |\xi|^{\frac{\beta}{\alpha}(1-\alpha)} e^{-it |\xi|^{\frac{\beta}{\alpha}}}
    - \frac{\sin(\pi(1-\alpha))}{\pi} t^{1-\alpha} \int_{0}^{\infty}
    \frac{|\xi|^{\frac{\beta}{\alpha}(1-\alpha)} e^{-rt |\xi|^{\frac{\beta}{\alpha}}} r^{\alpha}}
    {i^{\alpha} r^{2\alpha} - 2 r^{\alpha} \cos(\alpha \pi) + (-i)^{\alpha}} \, dr.
\end{align*}
\end{remark}
Next, we introduce some Littlewood-Paley decomposition for the tempered distribution, referring to \cite{Bahouri,Triebel}. Let $\chi_{-1}$ be a non-negative radial bump function such that $\widehat{\chi}_{-1}(\xi)=1$ for $|\xi|\leq \frac{3}{4}$ and $\widehat{\chi}_{-1}(\xi)=0$ for $|\xi|\geq \frac{4}{3}$. Define the function $\chi_{j}(x)=2^{jd}\chi_{-1}(2^{j}x)$ and consider the pseudo-differential operator $S_{j}$ with symbol $\chi_{j}$ defined pointwise as
$$
S_{j}(f)(x)=\chi_{j}\star f(x),\quad \text{for }f\in \mathcal{S}'.
$$
Define $\varphi_{j}(x)=\chi_{j}(x)-\chi_{j-1}(x)$ and consider the operator $\Delta_{j}$ defined as
$$
\Delta_{j}(f)(x)=S_{j}(f)(x)-S_{j-1}(f)(x)=\varphi_{j}\star f(x),\quad \text{for }f\in \mathcal{S}'.
$$
Hence, we have that
\begin{align*}
f&=\Delta_{-1}(f)+\sum_{j=0}^{\infty}\Delta_{j}(f),\quad f\in \mathcal{S}',\\
f&=\sum_{j\in \mathbb{Z}}\Delta_{j}(f),\quad f\in \mathcal{S}'_{h},
\end{align*}
where $\mathcal{S}'_{h}$ is the dual space of $\dot{\mathcal{S}}(\mathbb{R}^{d})$.
For $s \in \mathbb{R}$ and $1 \leq p, q \leq \infty$, the homogeneous Besov space $\dot{B}^{s}_{p,q}(\mathbb{R}^{d})$ and the Triebel-Lizorkin space $\dot{F}^{s}_{p,q}(\mathbb{R}^{d})$ are defined as follows:
\[
\dot{B}^{s}_{p,q}(\mathbb{R}^{d}) = \left\{ f \in \mathcal{S}'_{h}(\mathbb{R}^{d}) : \left\| f \right\|_{\dot{B}^{s}_{p,q}} < \infty \right\},
\]
and
\[
\dot{F}^{s}_{p,q}(\mathbb{R}^{d}) = \left\{ f \in \mathcal{S}'_{h}(\mathbb{R}^{d}) : \left\| f \right\|_{\dot{F}^{s}_{p,q}} < \infty \right\},
\]
where
\begin{align*}
\left\|f\right\|_{\dot{B}^{s}_{p,q}}=\left\|\left\{2^{js}\left\|\Delta_{j}(f)\right\|_{L^{p}}\right\}_{j\in\mathbb{Z}}\right\|_{l^{q}}
\text{ and }\left\|f\right\|_{\dot{F}^{s}_{p,q}}=
\left\|\left\|\left\{2^{js}\left|\Delta_{j}(f)\right|\right\}_{j\in\mathbb{Z}}\right\|_{l^{q}}\right\|_{L^{p}}.
\end{align*}
The homogeneous H\"{o}lder space $\dot{C}^{s}(\mathbb{R}^{d})$ is defined as
$$
\dot{C}^{s}(\mathbb{R}^{d})=\big\{f\in\mathcal{S}'_{h}:\text{ }\sup_{j\in\mathbb{Z}}2^{js}\big\|\triangle_{j}f\big\|_{L^{\infty}}<\infty\big\}.
$$
In particular, $L^{p}(\mathbb{R}^{d})\approx\dot{F}^{0}_{p,2}(\mathbb{R}^{d})$, $\dot{H}^{s}_{p}(\mathbb{R}^{d})\approx\dot{F}^{s}_{p,2}(\mathbb{R}^{d})$ and $\dot{B}^{s}_{\infty,\infty}(\mathbb{R}^{d})\sim \dot{C}^{s}(\mathbb{R}^{d})$.

Moreover, according to the Littlewood-Paley theory, the Hardy space $\mathcal{H}^{p}(\mathbb{R}^{d})$ for $0<p<\infty$ can be defined as
$$
\mathcal{H}^{p}(\mathbb{R}^{d})=\bigg\{w\in\mathcal{S}':\ \big\|w^{++}\big\|_{L^{p}}<\infty\bigg\},\ \text{where } w^{++}=\sup_{0<t<\infty}|\chi_{t}\star w|,\ \chi_{t}(x)=t^{-d}\chi(x/t),
$$
and $\mathcal{H}^{p}(\mathbb{R}^{d})=\dot{F}^{0}_{p,2}(\mathbb{R}^{d})$. In particular, $\mathcal{H}^{p}\sim L^{p}$ for $1<p<\infty$, but for $0<p<1$, $\mathcal{H}^{p}$ is only a Fréchet space, not a Banach space. Moreover,
$$
\mathcal{H}^{p}(\mathbb{R}^{d})^{\star}=\dot{C}^{d(\frac{1}{p}-1)}(\mathbb{R}^{d}),\quad \text{and}\quad \mathcal{H}^{1}(\mathbb{R}^{d})^{\star}=BMO(\mathbb{R}^{d}),
$$
and
$$
\dot{B}^{s}_{p,\min\{p,q\}}(\mathbb{R}^{d})\hookrightarrow\dot{F}^{s}_{p,q}(\mathbb{R}^{d})\hookrightarrow
\dot{B}^{s}_{p,\max\{p,q\}}(\mathbb{R}^{d}),\text{ for }0<p,q<\infty.
$$
For more details on Hardy spaces, we refer to \cite{Grafakos}.

Next we introduce some properties about interpolation space, can be found in \cite{Triebel,Kenig}.
\begin{proposition}\label{prop:interpolation}
The following properties hold:

\begin{enumerate}[\rm(i)]
    \item Let \( A_{0} \) and \( A_{1} \) be Banach spaces, and \( A_{0} \cap A_{1} \) be dense in both \( A_{0} \) and \( A_{1} \), while \( A_{0}^{\star} \cap A_{1}^{\star} \) is dense in both \( A_{0}^{\star} \) and \( A_{1}^{\star} \). Then,
    \[
    \left( BMO(A_{0}), L^{p}(A_{1}) \right)_{\theta} = L^{q} \left( (A_{0}, A_{1})_{\theta} \right) \quad \text{for } 1 < p < \infty \quad \text{and} \quad \frac{1}{q} = \frac{\theta}{p}.
    \]

    \item Let \( B_{1}, B_{2}, B_{3}, B_{4} \) be Banach spaces such that \( B_{1} \cap B_{2} \) is dense in both \( B_{2} \) and \( B_{3} \). For \( 0 < \theta_{1}, \theta_{2} < 1 \), and \( (B_{2}, B_{4})_{\theta_{1}} = B_{3} \), \( (B_{1}, B_{3})_{\theta_{2}} = B_{2} \), then we have \( (B_{1}, B_{4})_{\eta} = B_{3} \), where \( \eta = \frac{\theta_{1} \theta_{2}}{1 - \theta_{2} + \theta_{2}^{2}} \).
\end{enumerate}
\end{proposition}

Using Proposition \ref{prop:interpolation} and combining the Riesz interpolation
\[
\left( L^{p}_{x}L^{q}_{t}, L^{r}_{x}L^{\infty}_{t} \right)_{\theta} = L_{x}^{m}L^{s}_{t}, \quad \frac{1}{m} = \frac{\theta}{p} + \frac{1 - \theta}{r} \text{ and } \frac{1}{s} = \frac{\theta}{q},
\]
Kenig \cite{Kenig} obtains the following interpolation result:
\begin{equation}\label{Kenig interpolation}
\left( BMO_{x}L_{t}^{2}, L^{p}_{x}L^{\infty}_{t} \right)_{\theta} = L^{r}_{x}L^{s}_{t}, \quad 1 < p < \infty, \quad \frac{1}{r} = \frac{\theta}{p} \text{ and } \frac{1}{s} = \frac{1 - \theta}{2}.
\end{equation}

Moreover, Kenig \cite{Kenig} also get the following fractional Leibniz rule which is usuful to estimate the nonlinear term.
\begin{proposition}\label{Leibnitz rule}
Let $d=1$, then following estimate holds:

\item{\rm(i)}
$$
\left\|D^{\vartheta}F(f)\right\|_{L^{p}_{x}L^{q}_{T}}\lesssim \left\|F'(f)\right\|_{L^{p_{1}}_{x}L^{q_{1}}_{T}}\left\|D^{\vartheta}f\right\|_{L^{p_{2}}_{x}L^{q_{2}}_{T}},\quad
\frac{1}{p}=\frac{1}{p_{1}}+\frac{1}{p_{2}},\frac{1}{q}=\frac{1}{q_{1}}+\frac{1}{q_{2}},
$$
where $\vartheta\in(0,1)$, $1<p,q,p_{1},p_{2},q_{2}<\infty$ and $1<q_{1}\leq \infty$.

\item{\rm(ii)}
$$
\left\|D^{\vartheta}(fg)-(D^{\vartheta}f)g-f(D^{\vartheta}g)\right\|_{L^{p}_{x}L^{q}_{T}}\lesssim \left\|D^{\vartheta_{1}}f\right\|_{L^{p_{1}}_{x}L^{q_{1}}_{T}}\left\|D^{\vartheta_{2}}g\right\|_{L^{p_{2}}_{x}L^{q_{2}}_{T}},
\frac{1}{p}=\frac{1}{p_{1}}+\frac{1}{p_{2}},\frac{1}{q}=\frac{1}{q_{1}}+\frac{1}{q_{2}},
$$
where $\vartheta\in(0,1)$, $\vartheta=\vartheta_{1}+\vartheta_{2}$, $\vartheta_{1},\vartheta_{2}\in [0,\vartheta]$, $1<p,q,p_{1},q_{1},p_{2},q_{2}<\infty$. Moreover, if $\vartheta_{1}=0$, the $q_{1}=\infty$ is also allowed.
\end{proposition}
\begin{remark}\label{A.S.I.E}
 Due to the limitations imposed by interpolation, as discussed in Kenig \cite{Kenig}, neither $p_{1}$ nor $p_{2}$ can take the value of $\infty$. For example, let $m$ be an odd number, $s \in (0,1)$, $0 < \theta < 1$, then it is not possible to obtain the estimate:
\begin{align*}
\left\| D^{s} |f|^{m} \right\|_{L^{\frac{2m}{m-1}}_{x} L^{2}_{T}} \lesssim
\left\||f|^{m-1} \right\|_{L^{\frac{2m}{m-1}}_{x} L^{\infty}_{T}}
\left\| D^{s} f \right\|_{L^{\infty}_{x} L^{2}_{T}}.
\end{align*}
However, using Proposition \ref{Leibnitz rule}, we can assert that
\begin{align}\label{edge equation}
\left\| D^{s} |f|^{m} \right\|_{L^{\frac{2m}{m-1}}_{x} L^{2}_{T}} \lesssim
\left\||f|\right\|^{m-2}_{L^{2m}_{x} L^{\infty}_{T}}
\left\| D^{\theta s} f \right\|_{L^{\frac{2m}{1-\theta}}_{x} L^{\frac{2}{\theta}}_{T}}
\left\| D^{(1-\theta) s} f \right\|_{L^{\frac{2m}{\theta}}_{x} L^{\frac{2}{1-\theta}}_{T}}
+ \left\||f|\right\|^{m-1}_{L^{2m}_{x} L^{\infty}_{T}}
\left\| D^{s} f \right\|_{L^{\infty}_{x} L^{2}_{T}}.
\end{align}
Indeed, using H\"{o}lder's inequality and Proposition \ref{Leibnitz rule}, we have
\begin{align*}
&\left\| D^{s} |f|^{m} \right\|_{L^{\frac{2m}{m-1}}_{x} L^{2}_{T}}\\
&\lesssim\left\| (D^{s} |f|^{m-2})(f\bar{f}) \right\|_{L^{\frac{2m}{m-1}}_{x} L^{2}_{T}} +
\left\| |f|^{m-2} D^{s}(f\bar{f}) \right\|_{L^{\frac{2m}{m-1}}_{x} L^{2}_{T}} +
\left\| D^{s} |f|^{m-2} \right\|_{L^{\frac{2m}{m-3}}_{x} L^{2}_{T}}
\left\| f\bar{f} \right\|_{L^{m}_{x} L^{\infty}_{T}} \\
&\lesssim
\left\| D^{s} |f|^{m-2} \right\|_{L^{\frac{2m}{m-3}}_{x} L^{2}_{T}}
\left\| f\bar{f} \right\|_{L^{m}_{x} L^{\infty}_{T}} +
\left\||f|\right\|^{m-2}_{L^{2m}_{x} L^{\infty}_{T}}
\left\| D^{s} (f\bar{f}) \right\|_{L^{2m}_{x} L^{2}_{T}}.
\end{align*}
By continuing the above process for $\left\| D^{s} |f|^{m-2} \right\|_{L^{\frac{2m}{m-3}}_{x} L^{2}_{T}}$, we obtain
\begin{align*}
\left\| D^{s} |f|^{m-2} \right\|_{L^{\frac{2m}{m-3}}_{x} L^{2}_{T}} &\lesssim
\left\| D^{s} |f|^{m-4} \right\|_{L^{\frac{2m}{m-5}}_{x} L^{2}_{T}}
\left\| f\bar{f} \right\|_{L^{m}_{x} L^{\infty}_{T}} +
\left\||f|\right\|^{m-4}_{L^{2m}_{x} L^{\infty}_{T}}
\left\| D^{s} (f\bar{f}) \right\|_{L^{2m}_{x} L^{2}_{T}}.
\end{align*}
Hence, we obtain
\begin{align*}
\left\| D^{s} |f|^{m} \right\|_{L^{\frac{2m}{m-1}}_{x} L^{2}_{T}} &\lesssim
\left\| D^{s} |f|^{m-2} \right\|_{L^{\frac{2m}{m-3}}_{x} L^{2}_{T}}
\left\| f \right\|_{L^{2m}_{x} L^{\infty}_{T}} +
\left\||f|\right\|^{m-2}_{L^{2m}_{x} L^{\infty}_{T}}
\left\| D^{s} (f\bar{f}) \right\|_{L^{2m}_{x} L^{2}_{T}} \\
&\lesssim
\left\| D^{s} |f|^{m-4} \right\|_{L^{\frac{2m}{m-5}}_{x} L^{2}_{T}}
\left\| f\bar{f} \right\|^{2}_{L^{m}_{x} L^{\infty}_{T}} +
\left\||f|\right\|^{m-2}_{L^{2m}_{x} L^{\infty}_{T}}
\left\| D^{s} (f\bar{f}) \right\|_{L^{2m}_{x} L^{2}_{T}} \\
&\quad \vdots \\
&\lesssim
\left\||f|\right\|^{m-2}_{L^{2m}_{x} L^{\infty}_{T}}
\left\| D^{s} (f\bar{f}) \right\|_{L^{2m}_{x} L^{2}_{T}}.
\end{align*}
Note that by using Proposition \ref{Leibnitz rule},
\begin{align*}
\left\| D^{s} (f\bar{f}) \right\|_{L^{2m}_{x} L^{2}_{T}}&\lesssim\left\|(D^{s}f)\bar{f}\right\|_{L^{2m}_{x} L^{2}_{T}}+\left\|\bar{f}(D^{s}f)\right\|_{L^{2m}_{x} L^{2}_{T}}+\big\| D^{\theta s} f \big\|_{L^{\frac{2m}{1-\theta}}_{x} L^{\frac{2}{\theta}}_{T}}
\big\| D^{(1-\theta) s} f \big\|_{L^{\frac{2m}{\theta}}_{x} L^{\frac{2}{1-\theta}}_{T}}\\
&\lesssim\big\|D^{s}f\big\|_{L^{\infty}_{x}L^{2}_{T}}\big\|f\big\|_{L^{2m}_{x}L^{\infty}_{T}}+\big\| D^{\theta s} f \big\|_{L^{\frac{2m}{1-\theta}}_{x} L^{\frac{2}{\theta}}_{T}}
\big\| D^{(1-\theta) s} f \big\|_{L^{\frac{2m}{\theta}}_{x} L^{\frac{2}{1-\theta}}_{T}},
\end{align*}
hence we get \eqref{edge equation}.
\end{remark}
Next we introduce the Bessel function $J_{\nu}(t)$, that is defined
$$
J_{\nu}(t)=\frac{\left(\frac{t}{2}\right)^{\nu}}{\Gamma\left(\nu+\frac{1}{2}\right)\Gamma\left(\frac{1}{2}\right)}\int_{-1}^{1}
e^{its}\left(1-s^{2}\right)^{\nu-\frac{1}{2}}ds,\quad \Re(\nu)>\frac{1}{2},t\geq 0.
$$
The following relation can be founded in \cite{Grafakos}. There exists the function $h(t)$ and constant $c_{d}$ such that
\begin{align}\label{Bessel function transform}
J_{\frac{d-2}{2}}(t)=c_{d}t^{\frac{d-2}{2}}\Re\left(e^{it}h(t)\right)
\end{align}
and the function $h(t)$ satisfies that
\begin{align}\label{h function property}
\big|\partial^{k}_{t}h(t)\big|\lesssim \left(1+t\right)^{-\frac{n-1}{2}-k}.
\end{align}
The following lemma, known as the Van der Corput lemma, plays a key role in the proof of dispersive estimates and can be found in \cite{Grafakos}.
\begin{lemma}[Van der Corput Lemma]\label{Van der Corput lemma}
Consider a real-valued smooth function $\omega$ defined on the interval $(a,b)$ that satisfies $|\omega^{(k)}(x)| \geq \lambda$. Then, for any function $\phi$ with an integrable derivative on $(a,b)$, we have
$$
\left| \int_{a}^{b} e^{i\lambda \omega(x)} \phi(x) \, dx \right| \lesssim \lambda^{-\frac{1}{k}} \left[ |\phi(b)| + \int_{a}^{b} |\phi'(x)| \, dx \right],
$$
under the conditions:
\begin{enumerate}
\item[\rm(i)] $k \geq 2$, or
\item[\rm(ii)] $k = 1$ and $\omega'(x)$ is monotonic.
\end{enumerate}
\end{lemma}
\section{Some Lemma}
In this section, we establish several lemmas that are crucial for the proofs presented later in this paper. For convience, we denote that $\delta=\beta/\alpha$.
\begin{lemma}\label{find Mild solution}
If $w$ satisfies E.q.\eqref{Eq:F-T-S-E}, then the solution $w(t,x)$ can be represented as:
\begin{align}\label{mild solution}
w(t,x) &= E_{\alpha,1}((-it)^{\alpha}D^{\beta})w_{0}(x) + \frac{1}{i^{\alpha}}
\int_{0}^{t}(t-s)^{\alpha-1}E_{\alpha,\alpha}((-i(t-s))^{\alpha}D^{\beta})g(w(s,x)) \, ds,
\end{align}
Moreover, under the assumption of neglecting constants, the function $u(t,x)$ approximates to:
\begin{align}\label{mild solution equivalence}
w(t,x) &\sim \left(e^{-itD^{\delta}} + \int_{0}^{\infty}
\frac{e^{-rtD^{\delta}}r^{\alpha-1}}
{i^{\alpha}r^{2\alpha} - 2r^{\alpha}\cos(\alpha\pi) + (-i)^{\alpha}} \, dr\right) w_{0}(x)\\\notag
&\quad + i^{-\alpha} \int_{0}^{t} \left(D^{\delta-\beta}
e^{-i(t-s)D^{\delta}} + \int_{0}^{\infty} \frac{D^{\delta-\beta}
e^{-r(t-s)D^{\delta}} r^\alpha}{i^{\alpha}r^{2\alpha} - 2r^\alpha \cos(\alpha\pi) + (-i)^{\alpha}} \, dr \right) g(w(s,x)) \, ds,
\end{align}
where each operator is defined as a pseudo-differential with the symbol:
\begin{align*}
&E_{\alpha,1}((-it)^{\alpha}|\xi|^{\beta}),
E_{\alpha,\alpha}((-i(t-s))^{\alpha}|\xi|^{\beta}),
e^{-it|\xi|^{\delta}},|\xi|^{\delta-\beta} e^{-i(t-s)|\xi|^{\delta}},\\
&\int_{0}^{\infty}\frac{e^{-rt|\xi|^{\delta}}r^{\alpha-1}}{i^{\alpha}r^{2\alpha} - 2r^{\alpha}\cos(\alpha\pi) + (-i)^{\alpha}} dr,
 \int_{0}^{\infty}\frac{|\xi|^{\delta-\beta} e^{-r(t-s)|\xi|^{\delta}}r^\alpha}{i^{\alpha}r^{2\alpha} - 2r^\alpha \cos(\alpha\pi) + (-i)^{\alpha}} dr,
\end{align*}
respective to each term.
\end{lemma}

\begin{proof}
By applying the Fourier transform on E.q.\ref{Eq:F-T-S-E}, we obtain:
\begin{align*}
\begin{cases}
i^{\alpha} \partial_{t}^{\alpha} \mathcal{F}w(t, \xi) - |\xi|^{\beta} \mathcal{F}w(t, \xi) + \mathcal{F}(g(w))(t, \xi) = 0 & \text{in } (0,\infty) \times \mathbb{R}^{d}, \\
\mathcal{F}w(0, \xi) = \mathcal{F}(w_0)(\xi) & \text{in } \mathbb{R}^{d},
\end{cases}
\end{align*}
By applying the Laplace transform, we have:
\begin{align*}
\mathcal{L}(\widehat{w}(s, \xi)) = \frac{s^{\alpha-1}}{s^{\alpha} - {(-i)}^{\alpha} |\xi|^{\beta}}\widehat{w}_0(\xi) + \frac{{(-i)}^{\alpha}}{s^{\alpha} - {(-i)}^{\alpha} |\xi|^{\beta}} \mathcal{L}(\mathcal{F}(g(w)))(s, \xi),
\end{align*}
Additionally, considering Proposition \ref{prop:Mittag-Leffer} and Remark \ref{change variable equaton of Mittag-Leffer function}, we derive that
\begin{align*}
\widehat{w}(t, \xi) &= E_{\alpha,1}({(-it)}^{\alpha} |\xi|^{\beta}) \widehat{w}_0(\xi) + i^{-\alpha} \int_{0}^{t} (t-\tau)^{\alpha-1} E_{\alpha,\alpha}({(-i(t-\tau))}^{\alpha} |\xi|^{\beta}) \mathcal{F}(g(w))(\tau, \xi) \, d\tau \\
&= \left(\frac{1}{\alpha} e^{-it |\xi|^{\delta}} - \frac{\sin(\alpha \pi)}{\pi} \int_{0}^{\infty} \frac{e^{-rt |\xi|^{\delta}} r^{\alpha-1}}{{i}^{\alpha} r^{2\alpha} - 2 r^{\alpha} \cos(\alpha \pi) + {(-i)}^{\alpha}} \, dr\right) \widehat{w}_0(\xi) \\
&\quad + i^{-\alpha} \int_{0}^{t} \frac{1}{\alpha} {(-i)}^{1-\alpha} |\xi|^{\delta-\beta} e^{-i(t-\tau) |\xi|^{\delta}} \mathcal{F}(g(w))(\tau, \xi) \, d\tau \\
&\quad - \frac{\sin(\pi(1-\alpha))}{\pi} \int_{0}^{t} \int_{0}^{\infty} \frac{|\xi|^{\delta-\beta} e^{-r(t-\tau)|\xi|^{\delta}} r^{\alpha}}{{i}^{\alpha} r^{2\alpha} - 2 r^{\alpha} \cos(\alpha \pi) + {(-i)}^{\alpha}} \, dr \, \mathcal{F}(g(w))(\tau, \xi) \, d\tau.
\end{align*}
By applying the inverse Fourier transform and neglecting constants, we obtain that
\begin{align*}
w(t, x) &= E_{\alpha,1}({(-it)}^{\alpha} D^{\beta}) w_0(x) + i^{-\alpha} \int_{0}^{t} (t-s)^{\alpha-1} E_{\alpha,\alpha}({(-i(t-s))}^{\alpha} D^{\beta}) g(w(s, x)) \, ds \\
&\sim \left(e^{-itD^{\delta}} + \int_{0}^{\infty} \frac{e^{-rtD^{\delta}} r^{\alpha-1}}{i^{\alpha} r^{2\alpha} - 2r^{\alpha} \cos(\alpha \pi) + {(-i)}^{\alpha}} \, dr\right) w_0(x) \\
&\quad + i^{-\alpha} \int_{0}^{t} \left(D^{\delta-\beta} e^{-i(t-s)D^{\delta}} + \int_{0}^{\infty} \frac{D^{\delta-\beta} e^{-r(t-s)D^{\delta}} r^\alpha}{i^{\alpha} r^{2\alpha} - 2r^\alpha \cos(\alpha \pi) + {(-i)}^{\alpha}} \, dr\right) g(w(s, x)) \, ds.
\end{align*}
\end{proof}
\begin{definition}\label{mild definition}
Let \(X\) be a Banach space and let \(0 < T \leq \infty\). If a measurable function \(w: [0, T] \rightarrow X\) satisfies the integral equation \eqref{mild solution}, then we call \(u\) a mild solution of Eq. \eqref{Eq:F-T-S-E}. In particular, if \(T = \infty\), then we denote \(u\) as a global mild solution of E.q.\eqref{Eq:F-T-S-E}.
\end{definition}

Based on Lemma \ref{find Mild solution}, we have transformed the two solution operators of Eq.\eqref{Eq:F-T-S-E}, namely \(E_{\alpha,1}({(-it)}^{\alpha} D^{\beta})\) and \(E_{\alpha,\alpha}({(-it)}^{\alpha} D^{\beta})\), into forms corresponding to the fractional heat kernel and the Schr\"{o}dinger kernel. It is noteworthy that for the Schr\"{o}dinger kernels $\exp(-itD^{\delta})$ and \(D^{\delta-\beta} \exp(-itD^{\delta})\), \(L^{p}-L^{q}\) estimates similar to those for the heat kernel cannot be established; see \cite{Miao,Y.Z. Yang}. To address this issue, we need to ues the smoothing effects related to dispersive equations for the Schr\"{o}dinger kernels. The following content can be referred to Kenig \cite[Theorem 4.1]{Kenig1}.

For \(b(x,t) \in L^{\infty}(\mathbb{R}^d \times \mathbb{R})\) and \(g \in \mathcal{S}(\mathbb{R}^d)\), we define
\begin{align}\label{dispersive integral}
W(t)g(x) = \int_{\Omega} e^{i(t\varphi(\xi) + x\xi)} b(x, \varphi(\xi)) \widehat{g}(\xi) \, d\xi,
\end{align}
where \(\Omega \subset \mathbb{R}^d\) is an open set, $\varphi\in C^{1}(\Omega)$ and $\nabla \varphi(\xi)\neq 0$ for $\xi\in\Omega$. Furthermore, if there exists $m\in\mathbb{N}$ such that for any \(r \in \mathbb{R}\), the following system of equations
\begin{align}
\begin{cases}
\varphi(x, \xi_2, \ldots, \xi_d) = r, \\
\varphi(\xi_1, x, \ldots, \xi_d) = r, \\
\vdots \\
\varphi(\xi_1, \xi_2, \ldots, \xi_{k-1}, x, \ldots, \xi_d) = r, \\
\vdots \\
\varphi(\xi_1, \xi_2, \ldots, \xi_{d-1}, x) = r
\end{cases}
\end{align}
has at most \(m\) solutions, then according to Kenig \cite{Kenig1}, we can establish the following smoothing effects.
\begin{align}\label{local smoothing effect}
\int_{|x|<\kappa}\int_{-\infty}^{\infty}\left|W(t)g(x)\right|^{2}dtdx\lesssim\kappa m\int_{\Omega}\frac{|\widehat{g}(\xi)|^{2}}{|\nabla\varphi(\xi)|}d\xi,\quad d\geq 2,
\end{align}
and
\begin{align}
\label{global smoothing effect}\sup_{x}\int_{-\infty}^{\infty}\left|W(t)g(x)\right|^{2}dt&\lesssim\int_{\Omega}
\frac{|\widehat{g}(\xi)|^{2}}{|\varphi'(\xi)|}d\xi,\quad d=1.\\\label{global smoothing effect 1}
\int_{-\infty}^{\infty}\left|W(t)g(x)\right|^{2}dt&=\int_{\Omega}
\frac{|\widehat{g}(\xi)|^{2}}{|\varphi'(\xi)|}d\xi,\quad d=1, \varphi \text{ is invertible and } b\equiv 1.
\end{align}
\begin{remark}
From the above smoothing effects theorem for $d=1$, we have
\begin{enumerate}
\item[\rm(i)]for the classical Schr\"{o}dinger equation solution operator \(e^{it\Delta}\), and the Korteweg-de Vries equation solution operator \(e^{it(-\Delta)^{\frac{3}{2}}}\), we can obtain the following smoothing effects \rm\cite{Kenig}:
\begin{align*}
\left\|\partial_{x}^{\frac{1}{2}}e^{it\partial_{x}^{2}}f\right\|_{L^{\infty}_{x}L^{2}_{t}} \lesssim \left\|f\right\|_{L^{2}_{x}},\quad
\left\|\partial_{x}e^{it\partial_{x}^{3}}f\right\|_{L^{\infty}_{x}L^{2}_{t}} \lesssim \left\|f\right\|_{L^{2}_{x}}.
\end{align*}
\item[\rm(ii)] For the operators \(e^{-itD^{\delta}}\) and \(D^{\delta-\beta} e^{-itD^{\delta}}\) in the sense of \(L^{\infty}_{x}L^{2}_{t}\), there is a loss of derivatives \(D^{\frac{\delta-1}{2}}\) and \(D^{\beta-\frac{\delta+1}{2}}\), respectively.
\end{enumerate}
\end{remark}
Therefore, for $\delta=\beta/\alpha> 1$ and $\varphi(\xi)=-|\xi|^{\delta}$, by using the above smoothing effects and standard Strichartz estimate method for Schr\"{o}dinger operator, refer to \cite{Bahouri} , we can get the following global smoothing effects.
\begin{lemma}\label{smoothing effects 1}
Let $2\beta>\delta+1$, $d=1$, we have the following smoothing estimate for the operator $e^{-itD^{\delta}}$ and $D^{\delta-\beta} e^{-itD^{\delta}}$,
\begin{align}
\label{smoothing 1}\left\|D^{\frac{\delta-1}{2}} e^{-itD^\delta} f\right\|_{L^\infty_x L^2_t} &\lesssim \|f\|_{L^2_x}, \\\label{smoothing 2}
\left\|D^{\beta-\frac{\delta+1}{2}} D^{\delta-\beta} e^{-itD^\delta} f\right\|_{L^\infty_x L^2_t} &\lesssim \|f\|_{L^2_x}, \\\label{smoothing 3}
\left\|D^{\frac{\delta-1}{2}} \int_{\mathbb{R}} e^{itD^\delta} f(t) \, dt\right\|_{L^2_x} &\lesssim \|f\|_{L^1_x L^2_t}, \\\label{smoothing 4}
\left\|D^{\delta-1} \int_{\mathbb{R}} e^{i(t-s)D^\delta} f(s) \, ds\right\|_{L^\infty_xL^{2}_{t}} &\lesssim \|f\|_{L^1_x L^2_t}, \\\label{smoothing 5}
\left\|D^{\beta-\frac{\delta+1}{2}} \int_{\mathbb{R}} D^{\delta-\beta} e^{itD^\delta} f(t) \, dt\right\|_{L^2_x} &\lesssim \|f\|_{L^1_x L^2_t}, \\\label{smoothing 6}
\left\|D^{\frac{\delta-1}{2}}\int_{0}^{t}  e^{-i(t-s)D^\delta} f(s) \, ds\right\|_{L^\infty_x L^2_t} &\lesssim \|f\|_{L^1_x L^2_t}, \\\label{smoothing 7}
\left\|D^{\beta-\frac{\delta+1}{2}} \int_{0}^{t} D^{\delta-\beta} e^{-i(t-s)D^\delta} f(s) \, ds\right\|_{L^\infty_x L^2_t} &\lesssim \|f\|_{L^1_x L^2_t},\\\label{smoothing 8}
\left\|D^{\frac{\delta-1}{2}}\int_{0}^{t}e^{-i(t-s)D^\delta} f(s) \, ds\right\|_{L^\infty_x L^2_T} &\lesssim \|f\|_{L^1_T L^2_x}, \\\label{smoothing 9}
\left\|D^{\beta-\frac{\delta+1}{2}} \int_{0}^{t} D^{\delta-\beta} e^{-i(t-s)D^\delta} f(s) \, ds\right\|_{L^\infty_x L^2_T} &\lesssim \|f\|_{L^1_T L^2_x},\\\label{smooth 1}
\left\|e^{-itD^{\delta}}f\right\|_{L^{\infty}_{t}L^{2}_{x}}&\lesssim\left\|f\right\|_{L^{2}_{x}},\\\label{smooth 2}
\left\|\int_{0}^{t} D^{\delta-\beta} e^{-i(t-s)D^\delta} f(s) \, ds\right\|_{L^\infty_T L^2_x}&\lesssim\left\|D^{\delta-\beta}f\right\|_{L^{1}_{T}L^{2}_{x}}.
\end{align}
\end{lemma}

\begin{proof}
According to the smoothing estimates \eqref{global smoothing effect} and \eqref{global smoothing effect 1}, the estimates \eqref{smoothing 1} and \eqref{smoothing 2} are straightforward. By using the standard dual method, refer to \cite{Bahouri}, we can derive the estimates \eqref{smoothing 3}, \eqref{smoothing 5}, \eqref{smoothing 6}, and \eqref{smoothing 7}.

Moreover, for a function $g \in \mathcal{S}(\mathbb{R} \times \mathbb{R})$ with $\|g\|_{L^1_x L^2_t} \leq 1$, by using the Cauchy-Schwarz inequality, we obtain
\begin{align*}
\left|\int_{\mathbb{R}} \left\langle D^{\delta-1} \int_{\mathbb{R}} e^{i(t-s)D^\delta} f(s) \, ds, g(t)\right\rangle dt \right| &=
\left|\int_{\mathbb{R}}\int_{\mathbb{R}}\int_{\mathbb{R}} D^{\delta-1}e^{i(t-s)D^\delta} f(s,x) \overline{g(t,x)} \, ds \, dx \, dt\right| \nonumber \\
&= \left|\left\langle\int_{\mathbb{R}} D^{\frac{\delta-1}{2}}e^{isD^\delta}f(s) \, ds, \int_{\mathbb{R}}D^{\frac{\delta-1}{2}}e^{itD^\delta}g(t) \, dt\right\rangle\right| \nonumber \\
&\lesssim \left\|D^{\frac{\delta-1}{2}} \int_{\mathbb{R}} e^{isD^\delta} f(s,\cdot) \, ds\right\|_{L^2_x}\left\|D^{\frac{\delta-1}{2}} \int_{\mathbb{R}} e^{itD^\delta} g(t,\cdot) \, dt\right\|_{L^2_x} \nonumber \\
&\lesssim \|f\|_{L^1_x L^2_t},
\end{align*}
which implies estimate \eqref{smoothing 4}. Furthermore, following methods akin to Kenig \cite{Kenig}, and noting that the operator $e^{itD^\delta}$ is an identity operator, by the Minkowski inequality, we find that
\begin{align*}
\left\|D^{\frac{\delta-1}{2}}\int_{0}^{t}e^{-i(t-s)D^\delta} f(s) \, ds\right\|_{L^\infty_x L^2_T} &= \left\|D^{\frac{\delta-1}{2}}\int_{0}^{T}e^{-i(t-s)D^\delta} f(s)\chi_{(0,t)}(s) \, ds\right\|_{L^\infty_x L^2_T} \nonumber \\
&\lesssim \int_{0}^{T}\left\|D^{\frac{\delta-1}{2}}e^{-i(t-s)D^\delta} f(s)\chi_{(0,t)}(s)\right\|_{L^\infty_x L^2_T} \, ds \nonumber \\
&\lesssim \int_{0}^{T}\left\|f(s)\right\|_{L^{2}_{x}} \, ds,
\end{align*}
thus proving estimate \eqref{smoothing 8}. The proof of \eqref{smoothing 9} follows identically to that of \eqref{smoothing 8}. Moreover, the estimates \eqref{smoothing 8} and \eqref{smoothing 9} are also founded in \cite{Grande}. Finally, using the Plancherel identity, the estimates \eqref{smooth 1}, \eqref{smooth 2} are obvious.
\end{proof}
\begin{lemma}
Let $\delta>1$, $\beta<2$, $2\beta>\delta+1$, d=1, $\gamma=\frac{\delta-1}{2}$, $\nu=\beta-\frac{\delta+1}{2}$, we have the following smoothing estimate:
\begin{align}
\label{smoothing 10}
\left\|D^{\gamma'}\int_{0}^{\infty}
\frac{e^{-rtD^{\delta}}r^{\alpha-1}}
{i^{\alpha}r^{2\alpha} - 2r^{\alpha}\cos(\alpha\pi) + (-i)^{\alpha}} \, drf(x)\right\|_{L^{\infty}_{x}L^{2}_{T}}&\lesssim T^{\frac{\gamma-\gamma'}{\delta}}\left\|f\right\|_{L^{2}_{x}}, 0<\gamma'<\gamma,\\\label{smoothing 11}
\left\|D^{\nu'}\int_{0}^{t}\int_{0}^{\infty} \frac{D^{\delta-\beta} e^{-r(t-s)D^{\delta}} r^\alpha}{i^{\alpha} r^{2\alpha}
- 2r^\alpha \cos(\alpha \pi) + {(-i)}^{\alpha}} \, drf(s,x)ds\right\|_{L^{\infty}_{x}L^{2}_{T}}&\lesssim T^{\frac{\nu-\nu'}{\delta}}\left\|f\right\|_{L^{1}_{T}L^{2}_{x}}, 0<\nu'<\nu,\\\label{smooth 3}
\left\|\int_{0}^{\infty}
\frac{e^{-rtD^{\delta}}r^{\alpha-1}}
{i^{\alpha}r^{2\alpha} - 2r^{\alpha}\cos(\alpha\pi) + (-i)^{\alpha}} \, drf(x)\right\|_{L^{\infty}_{T}L^{2}_{x}}&\lesssim\left\|f\right\|_{L^{2}_{x}},\\\label{smooth 4}
\left\|\int_{0}^{t}\int_{0}^{\infty} \frac{D^{\delta-\beta} e^{-r(t-s)D^{\delta}} r^\alpha}{i^{\alpha} r^{2\alpha}
- 2r^\alpha \cos(\alpha \pi) + {(-i)}^{\alpha}} \, drf(s,x)ds\right\|_{L^{\infty}_{T}L^{2}_{x}}&\lesssim T^{\frac{3\alpha-2}{4}}\left\|f\right\|_{L^{2}_{T,x}}.
\end{align}
\end{lemma}
\begin{proof}
By using the estimate of fractional heat kernel \cite{Miao} and the Minkowski inequality, we can get
\begin{align*}
&\left\|D^{\gamma'}\int_{0}^{\infty}
\frac{e^{-rtD^{\delta}}r^{\alpha-1}}
{i^{\alpha}r^{2\alpha} - 2r^{\alpha}\cos(\alpha\pi) + (-i)^{\alpha}} \, drf(x)\right\|_{L^{\infty}_{x}L^{2}_{T}}\\
&\lesssim
\left\|D^{\gamma'}\int_{0}^{\infty}
\frac{e^{-rtD^{\delta}}r^{\alpha-1}}
{i^{\alpha}r^{2\alpha} - 2r^{\alpha}\cos(\alpha\pi) + (-i)^{\alpha}} \, drf(x)\right\|_{L^{2}_{T}L^{\infty}_{x}}\\
&\lesssim\left\|\int_{0}^{\infty}
\frac{\left\|D^{\gamma'}e^{-rtD^{\delta}}f(\cdot)\right\|_{L^{\infty}_{x}}r^{\alpha-1}}
{i^{\alpha}r^{2\alpha} - 2r^{\alpha}\cos(\alpha\pi) + (-i)^{\alpha}} \, dr\right\|_{L^{2}_{T}}\\
&\lesssim \left\|t^{-\frac{\gamma'}{\delta}-\frac{1}{2\delta}}\int_{0}^{\infty}
\frac{r^{\alpha-1-\frac{\gamma'}{\delta}-\frac{1}{2\delta}}}
{i^{\alpha}r^{2\alpha} - 2r^{\alpha}\cos(\alpha\pi) + (-i)^{\alpha}} \, dr\left\|f\right\|_{L^{2}_{x}}\right\|_{L^{2}_{T}}\\
&\lesssim T^{\frac{\gamma-\gamma'}{\delta}}\left\|f\right\|_{L^{2}_{x}},
\end{align*}
where the integral
$$
\int_{0}^{\infty}
\frac{r^{\alpha-1-\frac{\gamma'}{\delta}-\frac{1}{2\delta}}}
{i^{\alpha}r^{2\alpha} - 2r^{\alpha}\cos(\alpha\pi) + (-i)^{\alpha}} \, dr
$$
is convergent under $\frac{1}{2}<\alpha<1$, this prove the \eqref{smoothing 10}. Similar, for $0<\nu'<\nu$, we have that
\begin{align*}
&\left\|D^{\nu'}\int_{0}^{t}\int_{0}^{\infty} \frac{D^{\delta-\beta} e^{-r(t-s)D^{\delta}} r^\alpha}{i^{\alpha} r^{2\alpha} - 2r^\alpha \cos(\alpha \pi) + {(-i)}^{\alpha}} \, drf(s,x)ds\right\|_{L^{\infty}_{x}L^{2}_{T}}\\
&\lesssim\left\|\int_{0}^{t}\int_{0}^{\infty}\frac{\left\|D^{\delta-\beta+\nu'} e^{-r(t-s)D^{\delta}}f(s,\cdot)\right\|_{L^{\infty}_{x}} r^\alpha}{i^{\alpha} r^{2\alpha} - 2r^\alpha \cos(\alpha \pi) + {(-i)}^{\alpha}} \, drds\right\|_{L^{2}_{T}}\\
&\lesssim\left\|\int_{0}^{t}\int_{0}^{\infty}\frac{(t-s)^{-\frac{\delta+\nu'-\beta}{\delta}-\frac{1}{2\delta}}
r^{\alpha-\frac{\delta+\nu'-\beta}{\delta}-\frac{1}{2\delta}}}{i^{\alpha} r^{2\alpha} - 2r^\alpha \cos(\alpha \pi) + {(-i)}^{\alpha}}dr\left\|f(s,\cdot)\right\|_{L^{2}_{x}}ds\right\|_{L^{2}_{T}}\\
&\lesssim T^{\frac{\nu-\nu'}{\delta}}\left\|f\right\|_{L^{1}_{T}L^{2}_{x}},
\end{align*}
where we use the Young inequality and the integral
$$
\int_{0}^{\infty}\frac{r^{\alpha-\frac{\delta+\nu'-\beta}{\delta}-\frac{1}{2\delta}}}{i^{\alpha} r^{2\alpha} - 2r^\alpha \cos(\alpha \pi) + {(-i)}^{\alpha}}dr
$$
is convergent.
Thus, we derive that \eqref{smoothing 11}. Finally, by using the estimate of fractional heat kernel, the estimates \eqref{smooth 3} is obvious. Moreover, note that condition $2\beta>\delta+1$ ensure that $2/3<\alpha<1$ and combine Plancherel identity, we obtain that
\begin{align*}
&\bigg\|\int_{0}^{t}\int_{0}^{\infty} \frac{D^{\delta-\beta} e^{-r(t-s)D^{\delta}} r^\alpha}{i^{\alpha} r^{2\alpha}
- 2r^\alpha \cos(\alpha \pi) + {(-i)}^{\alpha}} \, drf(s,x)ds\bigg\|_{L^{\infty}_{T}L^{2}_{x}}\\
&\lesssim\bigg\|\int_{0}^{t}\int_{0}^{\infty} \frac{|r^{\frac{1}{\delta}}(t-s)^{\frac{1}{\delta}}\xi|^{\delta-\frac{3\beta}{4}} e^{-r|(t-s)^{\frac{1}{\delta}}\xi|^{\delta}} r^{\frac{7\alpha}{4}-1}}{i^{\alpha} r^{2\alpha}
- 2r^\alpha \cos(\alpha \pi) + {(-i)}^{\alpha}} \, dr(t-s)^{\frac{3\alpha}{4}-1}\widehat{f}(s,\xi)ds\bigg\|_{L^{\infty}_{T}L^{2}_{\xi}}\\
&\lesssim\bigg\|\int_{0}^{t}(t-s)^{\frac{3\alpha}{4}-1}\big\|\hat{f}(s,\xi)\big\|_{L^{2}_{\xi}}ds\bigg\|_{L^{\infty}_{T}}\cdot
\bigg|\int_{0}^{\infty}\frac{r^{\frac{7\alpha}{4}-1}}{i^{\alpha} r^{2\alpha}
- 2r^\alpha \cos(\alpha \pi) + {(-i)}^{\alpha}}dr\bigg|\\
&\lesssim T^{\frac{3\alpha-2}{4}}\left\|f\right\|_{L^{2}_{T,x}},
\end{align*}
where we use $\sup_{x\in\mathbb{R}^{+}}|x|^{\eta}e^{-|x|^{\delta}}<\infty$ for any $\eta>0$. The proof is completed now.
\end{proof}
\begin{remark}
In \eqref{smoothing 11}, if $\nu'=\nu$, we derive that
In particular, by using the fractional heat kernel estimates and Minkowski inequatily, we have
\begin{align*}
&\bigg\|D^{\nu}\int_{0}^{t}\int_{0}^{\infty} \frac{D^{\delta-\beta} e^{-r(t-s)D^{\delta}} r^\alpha}{i^{\alpha} r^{2\alpha} - 2r^\alpha \cos(\alpha \pi) + {(-i)}^{\alpha}} \, drf(s,x)ds\bigg\|_{L^{\infty}_{x}L^{2}_{T}}\\
&\lesssim\bigg\|\int_{0}^{t}\int_{0}^{\infty}\frac{r^{\alpha-\frac{1}{2}}}{i^{\alpha} r^{2\alpha} - 2r^\alpha \cos(\alpha \pi)
+ {(-i)}^{\alpha}}dr(t-s)^{-\frac{1}{2}}\left\|f(s,\cdot)\right\|_{L^{2}_{x}}ds\bigg\|_{L^{2}_{T}}\\
&\lesssim\bigg\|\int_{0}^{t}(t-s)^{-\frac{1}{2}}\big\|f(s,\cdot)\big\|_{L^{2}_{x}}ds\bigg\|_{L^{2}_{T}},
\end{align*}
but because the limit of Hardy-Littlewood inequality, it is not be controlled by $\big\|f\big\|_{L^{1}_{T}L^{2}_{x}}$, hence the condition $\nu'<\nu$ is necessary in \eqref{smoothing 11}.
\end{remark}
To use the embedding properties of non-homogeneous Sobolev spaces, we must prove the following estimates:
\begin{lemma}
Let $\delta>1$, $\beta<2$, $2\beta>\delta+1$, d=1, $\gamma=\frac{\delta-1}{2}$, $\nu=\beta-\frac{\delta+1}{2}$, the following estimates hold:
\begin{align}
\label{smoothing 12}\left\|e^{-itD^{\delta}}f\right\|_{L^{\infty}_{x}L^{2}_{T}}&\lesssim\left(T^{\frac{1}{2}}+1\right) \|f\|_{L^{2}_{x}},\\\label{smoothing 13}
\left\|\int_{0}^{\infty}
\frac{e^{-rtD^{\delta}}r^{\alpha-1}}
{i^{\alpha}r^{2\alpha} - 2r^{\alpha}\cos(\alpha\pi) + (-i)^{\alpha}} \, drf(\cdot)\right\|_{L^{\infty}_{x}L^{2}_{T}}&\lesssim\left(T^{\frac{\gamma-\gamma'}{\delta}}+T^{\frac{1}{2}}\right) \|f\|_{L^{2}_{x}}, 0<\gamma'<\gamma,\\\label{smoothing 14}
\left\|\int_{0}^{t} D^{\delta-\beta} e^{-i(t-s)D^\delta} f(s,\cdot) \, ds\right\|_{L^\infty_x L^2_t} &\lesssim\left(T^{\frac{1}{2}}+1\right)T^{\frac{1}{2}} \|f\|_{L^{2}_{x,T}},\\\label{smoothing 15}
\left\|\int_{0}^{t}\int_{0}^{\infty} \frac{D^{\delta-\beta} e^{-r(t-s)D^{\delta}} r^\alpha}{i^{\alpha} r^{2\alpha}
- 2r^\alpha \cos(\alpha \pi) + {(-i)}^{\alpha}} \, drf(s,\cdot)ds\right\|_{L^{\infty}_{x}L^{2}_{T}}&\lesssim\left(T^{\frac{\nu}{\delta}}+T^{\frac{\nu-\nu'}{\delta}}\right)T^{\frac{1}{2}}
\left\|f\right\|_{L^{2}_{x,T}}.
\end{align}
\end{lemma}

\begin{proof}
We choose $\widehat{\chi}\in C^{\infty}_{c}(\mathbb{R})$, hence we have $f=\chi f+\left(1-\chi\right)f$. We derive that
\begin{align*}
\left\|e^{-itD^{\delta}}f\right\|_{L^{\infty}_{x}L^{2}_{T}}&\lesssim\left\|e^{-itD^{\delta}}\left(\chi f+\left(1-\chi\right)f\right)\right\|_{L^{\infty}_{x}L^{2}_{T}}\\
&\lesssim\left\|e^{-itD^{\delta}}\chi f\right\|_{L^{\infty}_{x}L^{2}_{T}}+\left\|e^{-itD^{\delta}}(1-\chi)f\right\|_{L^{\infty}_{x}L^{2}_{T}}\\
&\lesssim\left\|e^{-it|\xi|^{\delta}}\widehat{\chi} \widehat{f}\right\|_{L^{2}_{T}L^{1}_{\xi}}+ \left\|D^{\frac{\delta-1}{2}}e^{-itD^{\delta}}D^{-\frac{\delta-1}{2}}\left(1-\chi\right)f\right\|_{L^{\infty}_{x}L^{2}_{T}}\\
&\lesssim T^{\frac{1}{2}}\left\|f\right\|_{L^{2}_{x}}+\left\|D^{-\frac{\delta-1}{2}}\left(1-\chi\right)f\right\|_{L^{2}_{x}}\\
&\lesssim\left(T^{\frac{1}{2}}+1\right) \|f\|_{L^{2}_{x}},
\end{align*}
where we use the Hausdorff-Young inequality, H\"{o}lder inequality, Plancherel identity the smoothing effect \eqref{smoothing 1}, and the symbol $|\xi|^{-\delta+1}\widehat{\chi}^{c}$ is a $L^{2}$ H\"{o}rmander multiplier, can refer to \cite{Triebel}. The estimate \eqref{smoothing 14} is identical and we omit the details.

Moreover, we have that
\begin{align*}
&\left\|\int_{0}^{\infty}
\frac{e^{-rtD^{\delta}}r^{\alpha-1}}
{i^{\alpha}r^{2\alpha} - 2r^{\alpha}\cos(\alpha\pi) + (-i)^{\alpha}} \, drf(x)\right\|_{L^{\infty}_{x}L^{2}_{T}}\\
&\lesssim\left\|\int_{0}^{\infty}
\frac{e^{-rt|\xi|^{\delta}}\widehat{\chi}(\xi) \widehat{f}(\xi)r^{\alpha-1}}
{i^{\alpha}r^{2\alpha} - 2r^{\alpha}\cos(\alpha\pi) + (-i)^{\alpha}} \, dr\right\|_{L^{2}_{T}L^{1}_{\xi}}\\
&\qquad+\left\|D^{\gamma'}\int_{0}^{\infty}
\frac{e^{-rtD^{\delta}}r^{\alpha-1}}
{i^{\alpha}r^{2\alpha} - 2r^{\alpha}\cos(\alpha\pi) + (-i)^{\alpha}} \, drD^{-\gamma'}(1-\chi)f(x)\right\|_{L^{\infty}_{x}L^{2}_{T}}\\
&\lesssim \left(T^{\frac{\gamma-\gamma'}{\delta}}+T^{\frac{1}{2}}\right) \|f\|_{L^{2}_{x}},
\end{align*}
and
\begin{align*}
&\left\|\int_{0}^{t}\int_{0}^{\infty} \frac{D^{\delta-\beta} e^{-r(t-s)D^{\delta}} r^\alpha}{i^{\alpha} r^{2\alpha}
- 2r^\alpha \cos(\alpha \pi) + {(-i)}^{\alpha}} \, drf(s,\cdot)ds\right\|_{L^{\infty}_{x}L^{2}_{T}}\\
&\lesssim\left\|\int_{0}^{t}\int_{0}^{\infty} \frac{D^{\delta-\beta} e^{-r(t-s)D^{\delta}}\chi f(s,\cdot) r^\alpha}{i^{\alpha} r^{2\alpha}
- 2r^\alpha \cos(\alpha \pi) + {(-i)}^{\alpha}} \, drds\right\|_{L^{2}_{T}L^{\infty}_{x}}\\
&\quad +\left\|D^{\nu}\int_{0}^{t}\int_{0}^{\infty} \frac{D^{\delta-\beta} e^{-r(t-s)D^{\delta}} r^\alpha}{i^{\alpha} r^{2\alpha}
- 2r^\alpha \cos(\alpha \pi) + {(-i)}^{\alpha}} \, drD^{-\nu}\left(1-\chi\right) f(s,\cdot)ds\right\|_{L^{\infty}_{x}L^{2}_{T}}\\
&\lesssim\left\|\int_{0}^{t}\int_{0}^{\infty}\frac{r^{\alpha-\frac{\delta-\beta}{\delta}-\frac{1}{2\delta}}
(t-s)^{-\frac{\delta-\beta}{\delta}-\frac{1}{2\delta}}
\left\|f(s,\cdot)\right\|_{L^{2}_{x}}}{i^{\alpha} r^{2\alpha}
- 2r^\alpha \cos(\alpha \pi) + {(-i)}^{\alpha}}drds\right\|_{L^{2}_{T}}+T^{\frac{\nu-\nu'}{\delta}}\left\|D^{-\nu'}\left(1-\chi\right) f(s,\cdot)\right\|_{L^{1}_{T}L^{2}_{x}}\\
&\lesssim \left(T^{\frac{\nu}{\delta}}+T^{\frac{\nu-\nu'}{\delta}}\right)\left\|f\right\|_{L^{1}_{T}L^{2}_{x}}
\end{align*}
where we use the smoothing estimate \eqref{smoothing 10}, \eqref{smoothing 11}, the H\"{o}rmander multiplier theory, and the integral
$$
\int_{0}^{\infty}
\frac{r^{\alpha-1}}
{i^{\alpha}r^{2\alpha} - 2r^{\alpha}\cos(\alpha\pi) + (-i)^{\alpha}} \, dr,
\int_{0}^{\infty}
\frac{r^{\alpha-\frac{\delta-\beta}{\delta}-\frac{1}{2\delta}}}
{i^{\alpha}r^{2\alpha} - 2r^{\alpha}\cos(\alpha\pi) + (-i)^{\alpha}} \, dr
$$
are convergent.
\end{proof}
Next, we need the following $L^{p}_{x}L^{\infty}_{T}$ estimates.
\begin{lemma}\label{maximum function estimate}
Let $d=1$, $\delta>1$, $\frac{1}{2}<\alpha<1$, $\beta<2$, $\frac{1}{4}\leq s<\frac{1}{2}$, $\frac{1}{p_{s}}=\frac{1}{2}-s$, then we have the following estimates:
\begin{align}
\label{smoothing 16}
\left\|e^{-itD^{\delta}}f\right\|_{L^{p_{s}}_{x}L^{\infty}_{t}}&\lesssim\left\|D^{s}f\right\|_{L^{2}_{x}},\\\label{smoothing 17}
\left\|\int_{0}^{\infty}
\frac{e^{-rtD^{\delta}}r^{\alpha-1}}
{i^{\alpha}r^{2\alpha} - 2r^{\alpha}\cos(\alpha\pi) + (-i)^{\alpha}} \, drf(\cdot)\right\|_{L^{p_{s}}_{x}L^{\infty}_{t}}&\lesssim\left\|D^{s}f\right\|_{L^{2}_{x}},\\\label{smoothing 18}
\left\|\int_{0}^{t} D^{\delta-\beta} e^{-i(t-s)D^\delta} f(s,\cdot) \, ds\right\|_{L^{p_{s}}_{x}L^{\infty}_{T}}&\lesssim T^{\frac{1}{2}}\left\|D^{\delta-\beta+s}f\right\|_{L^{2}_{T,x}},\\\label{smoothing 19}
\left\|\int_{0}^{t}\int_{0}^{\infty} \frac{D^{\delta-\beta} e^{-r(t-\tau)D^{\delta}} r^\alpha}{i^{\alpha} r^{2\alpha}
- 2r^\alpha \cos(\alpha \pi) + {(-i)}^{\alpha}} \, drf(\tau,\cdot)d\tau\right\|_{L^{p_{s}}_{x}L^{\infty}_{T}}&\lesssim T^{\alpha-\frac{1}{2}}\left\|D^{s}f\right\|_{L^{2}_{T,x}}
\end{align}
\end{lemma}
\begin{proof}

The proof of \eqref{smoothing 16} follows Kenig \cite{Kenig}. We define the following integral
\begin{align}\label{Kenig integral}
I^{t}_{\vartheta}=\int_{\mathbb{R}}e^{i\left(t|\xi|^{\delta}+x\xi\right)}\frac{d\xi}{|\xi|^{\vartheta}},\quad
\frac{1}{2}\leq\vartheta<1.
\end{align}
Kenig \cite{Kenig} proved that for $\delta=3$, $\left|I^{t}_{\vartheta}\right|\lesssim \left|x\right|^{\vartheta-1}$. Indeed, we can claim for any $\delta>1$, $\left|I^{t}_{\vartheta}\right|\lesssim \left|x\right|^{\vartheta-1}$.

Indeed, for $t=0$, the conclusion is obvious, see \cite{Grafakos}. For $t\neq 0$, by the scaling transform, we only need to verify the case $t=1$. Let us denote the following regions:
\begin{align*}
\Omega_{1}&=\left\{\xi\in\mathbb{R}:|\xi|\lesssim\max\{2,|x|^{-1}\}\right\},\\
\Omega_{2}&=\left\{\xi\in\Omega_{1}^{c}:\left|\delta|\xi|^{\delta-1}+x\right|\leq\frac{1}{\delta-1}|x|\right\},\\
\Omega_{3}&=\mathbb{R}\setminus\left(\Omega_{1}\cup\Omega_{2}\right).
\end{align*}
Thus, we can write $I^{1}_{\vartheta}$ as
$$
I^{1}_{\vartheta}=\sum_{j=1}^{3}\int_{\Omega_{j}}e^{i\left(|\xi|^{\delta}+x\xi\right)}\frac{d\xi}{|\xi|^{\vartheta}}
=\sum_{j=1}^{3}I_{\vartheta,j}^{1}.
$$
First, $\left|I_{\vartheta,1}^{1}\right|\lesssim \left|x\right|^{\vartheta-1}$ is obvious. If $\xi\in\Omega_{2}$, there exist constants $c_{1}, c_{2}$ such that $c_{1}\left|x\right|\leq \left|\xi\right|^{\delta-1}\leq c_{2} \left|x\right|$, hence by the
Lemma \ref{Van der Corput lemma}, we have that
\begin{align*}
\left|\int_{\Omega_{2}}e^{i\left(|\xi|^{\delta}+x\xi\right)}\frac{d\xi}{|\xi|^{\vartheta}}\right|&\lesssim \left(\left|\xi\right|^{\delta-2}\right)^{-\frac{1}{2}}
\left(\left\||\xi|^{-\vartheta}\chi_{|\xi|^{\delta-1}\sim|x|}(\cdot)\right\|_{\infty}
+\left\||\xi|^{-\vartheta-1}\chi_{|\xi|^{\delta-1}\sim|x|}(\cdot)\right\|_{1}\right)\\
&\lesssim \left|x\right|^{\vartheta-1}.
\end{align*}
If $\xi\in\Omega_{3}$, then we get $\left||\xi|^{\delta-1}+x\right|\gtrsim \left||\xi|^{\delta-1}+|x|\right|$, hence by integration by parts, we estimate that
\begin{align*}
\left|\int_{\Omega_{3}}e^{i\left(|\xi|^{\delta}+x\xi\right)}\frac{d\xi}{|\xi|^{\vartheta}}\right|&=
\left|\int_{\Omega_{3}}e^{i\left(|\xi|^{\delta}+x\xi\right)}\frac{d}{d\xi}
\left(\frac{1}{\delta|\xi|^{\delta-1}+x}\frac{1}{|\xi|^{\vartheta}}\right)d\xi\right|\\
&\lesssim\left|\int_{\Omega_{3}}e^{i\left(|\xi|^{\delta}+x\xi\right)}
\left(\frac{|\xi|^{\delta-1}}{\left(\delta|\xi|^{\delta-1}+x\right)^{2}}\frac{1}{|\xi|^{\vartheta+1}}
+\frac{1}{\delta|\xi|^{\delta-1}+x}\frac{1}{|\xi|^{\vartheta+1}}\right)d\xi\right|\\
&\lesssim\left|x\right|^{\vartheta-1}.
\end{align*}
Moreover, by the standard dual method, \eqref{smoothing 16} is equivalent to the following estimate
\begin{align}\label{dual integral}
A = \left\| \int_{\mathbb{R}} D^{-s} e^{itD^{\delta}} f(t, \cdot) dt \right\|_{L^2_x} \lesssim \left\| f \right\|_{L^{p_{s}'}_x L^1_t}.
\end{align}
In fact,
\begin{align*}
A^2 &= \int_{\mathbb{R}} \left( \int_{\mathbb{R}} D^{-s} e^{itD^{\delta}} f(t,x) dt \right) \overline{\left( \int_{\mathbb{R}} D^{-s} e^{it'D^{\delta}} f(t',x) dt' \right)} dx \\
&= \int_{\mathbb{R}} \int_{\mathbb{R}} \left( \int_{\mathbb{R}} D^{-2s} e^{i(t-t')D^{\delta}} \overline{f(t',x)} dt' \right) f(t,x) dt dx \\
&\lesssim \left\| \int_{\mathbb{R}} D^{-2s} e^{i(t-t')D^{\delta}} f(t',x) dt' \right\|_{L^{p_{s}}_x L^\infty_t} \left\| f \right\|_{L^{p_{s}'}_x L^1_t} \\
&\lesssim \left\| \frac{1}{|x|^{1-2s}} \star \int_{\mathbb{R}} f(t', \cdot) dt' \right\|_{L^{p_{s}}_x} \left\| f \right\|_{L^{p_{s}'}_x L^1_t} \\
&\lesssim \left\| f \right\|_{L^{p_{s}'}_x L^1_t}^2,
\end{align*}
where we use the Hardy-Littlewood inequality, hence implying \eqref{smoothing 16}. Additionally, noting that $e^{itD^{\delta}}$ is an unity operator and by Minkowski inequality, we get
\begin{align*}
\left\| \int_{0}^{t} D^{\delta-\beta} e^{-i(t-s)D^\delta} f(s, \cdot) ds \right\|_{L^{p_{s}}_x L^\infty_T} &\lesssim \int_{0}^{T} \left\| D^{\delta-\beta+s} e^{isD^{\delta}} \chi_{(0,t)}(s) f(s, \cdot) \right\|_{L^2_x} ds \\
&\lesssim T^{\frac{1}{2}} \left\| D^{\delta-\beta+s} f \right\|_{L^2_{T,x}},
\end{align*}
hence, this implies that \eqref{smoothing 18}.

Next we prove the estimate \eqref{smoothing 17}. By Hardy-Littlewood inequality, \eqref{smoothing 17} is equivalent to verify
\begin{align}\label{heat estimate 1}
\left|\int_{\mathbb{R}}\int_{0}^{\infty}
\frac{|\xi|^{-s}e^{-rt|\xi|^{\delta}}r^{\alpha-1}}
{i^{\alpha}r^{2\alpha} - 2r^{\alpha}\cos(\alpha\pi) + (-i)^{\alpha}}e^{ix\xi}\, drd\xi\right|\lesssim \left|x\right|^{s-1}.
\end{align}
By scailing transform, we only prove \eqref{heat estimate 1} for $t=1$.

For $x$ is small ($|x|< 1$), we estimate
\begin{align*}
&\left|\int_{\mathbb{R}}\int_{0}^{\infty}
\frac{|\xi|^{-s}e^{-r|\xi|^{\delta}}r^{\alpha-1}}
{i^{\alpha}r^{2\alpha} - 2r^{\alpha}\cos(\alpha\pi) + (-i)^{\alpha}}e^{ix\xi}\, drd\xi\right|\\
&\leq\left|\int_{\left|\xi\right|\leq 1}\int_{0}^{\infty}
\frac{|\xi|^{-s}e^{-r|\xi|^{\delta}}r^{\alpha-1}}
{i^{\alpha}r^{2\alpha} - 2r^{\alpha}\cos(\alpha\pi) + (-i)^{\alpha}}e^{ix\xi}\, drd\xi\right| +\left|\int_{\left|\xi\right|\geq 1}\int_{0}^{\infty}
\frac{|\xi|^{-s}e^{-r|\xi|^{\delta}}r^{\alpha-1}}
{i^{\alpha}r^{2\alpha} - 2r^{\alpha}\cos(\alpha\pi) + (-i)^{\alpha}}e^{ix\xi}\, drd\xi\right|\\
&\lesssim\left|\int_{|\xi|\leq 1}|\xi|^{-s}d\xi\right|+\left|\int_{|\xi|\geq 1}\int_{0}^{\infty}\frac{r^{\alpha-\frac{3}{4\delta}-1}}{i^{\alpha}r^{2\alpha}-2r^{\alpha}\cos(\alpha\pi)+(-i)^{\alpha}}dr|\xi|^{-s-\frac{3}{4}}
\sup_{(r,\xi)\in\mathbb{R}^{+}\times\mathbb{R}}\big(|r^{\frac{1}{\delta}}
\xi|^{\frac{3}{4}}e^{-|r^{\frac{1}{\delta}}\xi|^{\delta}}\big)d\xi\right|\\
&\lesssim\left|\int_{|\xi|\leq 1}|\xi|^{-s}d\xi\right|+\left|\int_{|\xi|\geq 1}|\xi|^{-s-\frac{3}{4}}d\xi\right|\\
&\lesssim 1\lesssim \left|x\right|^{s-1}.
\end{align*}
Furthermore, for $|x|>1$, note that $\sup_{x\in\mathbb{R}^{+}}x^{\eta}e^{-x}<\infty$ for any $\eta>0$,  using partial integration we have
\begin{align*}
&\left|\int_{\left|\xi\right|\geq 1}\int_{0}^{\infty}
\frac{|\xi|^{-s}e^{-r|\xi|^{\delta}}r^{\alpha-1}}
{i^{\alpha}r^{2\alpha} - 2r^{\alpha}\cos(\alpha\pi) + (-i)^{\alpha}}e^{ix\xi}\, drd\xi\right|\\
&=\left|\frac{1}{x}\int_{\left|\xi\right|\geq 1}e^{ix\xi}\frac{d}{d\xi}\int_{0}^{\infty}\frac{|\xi|^{-s}e^{-r|\xi|^{\delta}}r^{\alpha-1}}
{i^{\alpha}r^{2\alpha} - 2r^{\alpha}\cos(\alpha\pi) + (-i)^{\alpha}}\, drd\xi\right|\\
&\lesssim\left|\frac{1}{x}\int_{\left|\xi\right|\geq 1}e^{ix\xi}\int_{0}^{\infty}\left(\frac{|\xi|^{\delta-1-s}e^{-r|\xi|^{\delta}}r^{\alpha}}
{i^{\alpha}r^{2\alpha} - 2r^{\alpha}\cos(\alpha\pi) + (-i)^{\alpha}}+\frac{|\xi|^{-1-s}e^{-r|\xi|^{\delta}}r^{\alpha-1}}
{i^{\alpha}r^{2\alpha} - 2r^{\alpha}\cos(\alpha\pi) + (-i)^{\alpha}}\right)\, drd\xi\right|\\
&\lesssim\left|\frac{1}{x}\int_{\left|\xi\right|\geq 1}e^{ix\xi}\left|\xi\right|^{-s-1}d\xi\cdot\int_{0}^{\infty}\frac{r^{\alpha-1}}{i^{\alpha}r^{2\alpha} - 2r^{\alpha}\cos(\alpha\pi) + (-i)^{\alpha}}dr\right|\lesssim \frac{1}{|x|}\lesssim \left|x\right|^{s-1},
\end{align*}
and
\begin{align*}
&\left|\int_{\left|\xi\right|\leq 1}\int_{0}^{\infty}
\frac{|\xi|^{-s}e^{-r|\xi|^{\delta}}r^{\alpha-1}}
{i^{\alpha}r^{2\alpha} - 2r^{\alpha}\cos(\alpha\pi) + (-i)^{\alpha}}e^{ix\xi}\, drd\xi\right|\\
&\lesssim\left|\int_{\left|\xi\right|\leq \left|x\right|^{-1}}\int_{0}^{\infty}
\frac{|\xi|^{-s}e^{-r|\xi|^{\delta}}r^{\alpha-1}}
{i^{\alpha}r^{2\alpha} - 2r^{\alpha}\cos(\alpha\pi) + (-i)^{\alpha}}e^{ix\xi}\, drd\xi\right|\\
&\quad +
\left|\int_{\left|x\right|^{-1}\leq\left|\xi\right|\leq 1}\int_{0}^{\infty}
\frac{|\xi|^{-s}e^{-r|\xi|^{\delta}}r^{\alpha-1}}
{i^{\alpha}r^{2\alpha} - 2r^{\alpha}\cos(\alpha\pi) + (-i)^{\alpha}}e^{ix\xi}\, drd\xi\right|\\
&\lesssim \left|\int_{|\xi|\leq |x|^{-1}}\left|\xi\right|^{-s}d\xi\right|\\
&\quad+\left|\frac{1}{x}\int_{\left|x\right|^{-1}\leq\left|\xi\right|\leq 1}\int_{0}^{\infty}
\left(\frac{|\xi|^{\delta-s-1}e^{-r|\xi|^{\delta}}r^{\alpha}}
{i^{\alpha}r^{2\alpha} - 2r^{\alpha}\cos(\alpha\pi) + (-i)^{\alpha}}+\frac{|\xi|^{-s-1}e^{-r|\xi|^{\delta}}r^{\alpha-1}}
{i^{\alpha}r^{2\alpha} - 2r^{\alpha}\cos(\alpha\pi) + (-i)^{\alpha}}\right)e^{ix\xi}\, drd\xi\right|\\
&\lesssim \left|\int_{|\xi|\leq |x|^{-1}}\left|\xi\right|^{-s}d\xi\right|+\left|\frac{1}{x}\int_{|x|^{-1}\leq|\xi|\leq 1}\left|\xi\right|^{-s-1}d\xi\right|\lesssim \left|x\right|^{s-1}.
\end{align*}
from above, we have the estimate \eqref{smoothing 17}.

The proof of \eqref{smoothing 19} is similar to that of \eqref{smoothing 17}, hence, we present a brief calculation.

First, we prove that
\begin{align}\label{heat estimate 2}
\left|\int_{\mathbb{R}}\int_{0}^{\infty}\frac{\left|\xi\right|^{\delta-\beta-s}e^{-rt\left|\xi\right|^{\delta}}r^{\alpha}}
{i^{\alpha}r^{2\alpha} - 2r^{\alpha}\cos(\alpha\pi) + (-i)^{\alpha}}e^{ix\xi}drd\xi\right|\lesssim t^{\alpha-1}\left|x\right|^{s-1}.
\end{align}
By the scailing transform, the estimate \eqref{heat estimate 2} is only to verify for $t=1$. The process is similar to the estimate \eqref{heat estimate 1}.

For $x$ is small ($|x|<1$), note that $\frac{\beta}{4}+s<1$, the integral
$$
\int_{0}^{\infty}\frac{r^{\frac{7\alpha}{4}-1}}{i^{\alpha}r^{2\alpha} - 2r^{\alpha}\cos(\alpha\pi) + (-i)^{\alpha}}dr
$$
is convengent,
we have that
\begin{align*}
&\left|\int_{\mathbb{R}}\int_{0}^{\infty}\frac{\left|\xi\right|^{\delta-\beta-s}e^{-r\left|\xi\right|^{\delta}}r^{\alpha}}
{i^{\alpha}r^{2\alpha} - 2r^{\alpha}\cos(\alpha\pi) + (-i)^{\alpha}}e^{ix\xi}drd\xi\right|\\
&\lesssim \left|\int_{\mathbb{R}}\int_{0}^{\infty}\frac{\left|\xi\right|^{-\frac{\beta}{4}-s}
\left(\left|r^{\frac{1}{\delta}}\xi\right|^{\delta-\frac{3\beta}{4}}e^{-r\left|\xi\right|^{\delta}}\right)
r^{\frac{7\alpha}{4}-1}}
{i^{\alpha}r^{2\alpha} - 2r^{\alpha}\cos(\alpha\pi) + (-i)^{\alpha}}e^{ix\xi}drd\xi\right|\\
&\lesssim\left|\int_{\mathbb{R}}e^{ix\cdot\xi}|\xi|^{-\frac{\beta}{4}-s}d\xi\right|\cdot
\left|\sup_{(r,\xi)\in\mathbb{R}^{+}\times\mathbb{R}}
\bigg(|r^{\frac{1}{\delta}}\xi|^{\delta-\frac{3\beta}{4}}e^{-r\left|\xi\right|^{\delta}}\bigg)\right|\\
&\lesssim \left|x\right|^{\frac{\beta}{4}+s-1}\lesssim\left|x\right|^{s-1}.
\end{align*}
Moreover, for $|x|>1$, we also have that
\begin{align*}
&\left|\int_{\mathbb{R}}\int_{0}^{\infty}\frac{\left|\xi\right|^{\delta-\beta-s}e^{-r\left|\xi\right|^{\delta}}r^{\alpha}}
{i^{\alpha}r^{2\alpha} - 2r^{\alpha}\cos(\alpha\pi) + (-i)^{\alpha}}e^{ix\xi}drd\xi\right|\\
&\lesssim\left|\int_{\left|\xi\right|\leq \left|x\right|^{-1}}\int_{0}^{\infty}\frac{\left|\xi\right|^{\delta-\beta-s}e^{-r\left|\xi\right|^{\delta}}r^{\alpha}}
{i^{\alpha}r^{2\alpha} - 2r^{\alpha}\cos(\alpha\pi) + (-i)^{\alpha}}e^{ix\xi}drd\xi\right|\\
&\quad+\left|\int_{\left|x\right|^{-1}\leq\left|\xi\right|\leq 1}\int_{0}^{\infty}\frac{\left|\xi\right|^{\delta-\beta-s}e^{-r\left|\xi\right|^{\delta}}r^{\alpha}}
{i^{\alpha}r^{2\alpha} - 2r^{\alpha}\cos(\alpha\pi) + (-i)^{\alpha}}e^{ix\xi}drd\xi\right|\\
&\quad +\left|\int_{\left|\xi\right|\geq 1}\int_{0}^{\infty}\frac{\left|\xi\right|^{\delta-\beta-s}e^{-r\left|\xi\right|^{\delta}}r^{\alpha}}
{i^{\alpha}r^{2\alpha} - 2r^{\alpha}\cos(\alpha\pi) + (-i)^{\alpha}}e^{ix\xi}drd\xi\right|.
\end{align*}
Similar to the estimate in \eqref{smoothing 17}, we derive the estimate
\begin{align*}
&\left|\int_{\left|\xi\right|\geq 1}\int_{0}^{\infty}\frac{\left|\xi\right|^{\delta-\beta-s}e^{-r\left|\xi\right|^{\delta}}r^{\alpha}}
{i^{\alpha}r^{2\alpha} - 2r^{\alpha}\cos(\alpha\pi) + (-i)^{\alpha}}e^{ix\xi}drd\xi\right|\\
&\lesssim\left|\frac{1}{x}\int_{|\xi|\geq 1}e^{ix\cdot\xi}\int_{0}^{\infty}e^{-r|\xi|^{\delta}}\bigg(\frac{|\xi|^{2\delta-\beta-s-1}r^{\alpha+1}
+|\xi|^{\delta-\beta-s-1}r^{\alpha}}{i^{\alpha}r^{2\alpha}-2r^{\alpha}\cos(\alpha\pi)+(-i)^{\alpha}}\bigg)drd\xi\right|\\
&\lesssim\frac{1}{\left|x\right|}\int_{\left|\xi\right|\geq 1}\frac{d\xi}{\left|\xi\right|^{\beta+s+1}}\int_{0}^{\infty}\frac{r^{\alpha-1}}{i^{\alpha}r^{2\alpha} - 2r^{\alpha}\cos(\alpha\pi) + (-i)^{\alpha}}dr\cdot\sup_{(r,\xi)\in\mathbb{R}^{+}\times[-1,1]^{c}}\big(r^{2}|\xi|^{2\delta}e^{-r|\xi|^{\delta}}
+r|\xi|^{\delta}e^{-r|\xi|^{\delta}}\big)\\
&\lesssim \frac{1}{\left|x\right|}\lesssim\left|x\right|^{s-1},
\end{align*}
and by using Remark \ref{change variable equaton of Mittag-Leffer function}, we get
\begin{align*}
&\left|\int_{\left|\xi\right|\leq \left|x\right|^{-1}}\int_{0}^{\infty}\frac{\left|\xi\right|^{\delta-\beta-s}e^{-r\left|\xi\right|^{\delta}}r^{\alpha}}
{i^{\alpha}r^{2\alpha} - 2r^{\alpha}\cos(\alpha\pi) + (-i)^{\alpha}}e^{ix\xi}drd\xi\right|\\
&\sim\left|\int_{\left|\xi\right|\leq \left|x\right|^{-1}}|\xi|^{-s}
\left(E_{\alpha,\alpha}((-i)^{\alpha}|\xi|^{\beta})-|\xi|^{\delta-\beta}e^{-i|\xi|^{\delta}}\right)e^{ix\xi}d\xi\right|\\
&\lesssim\left|\int_{|\xi|\leq|x|^{-1}}|\xi|^{-s}d\xi\right|
\lesssim\left|x\right|^{s-1},
\end{align*}
and
\begin{align*}
&\left|\int_{\left|x\right|^{-1}\leq|\xi|\leq1}
\int_{0}^{\infty}\frac{\left|\xi\right|^{\delta-\beta-s}e^{-r\left|\xi\right|^{\delta}}r^{\alpha}}
{i^{\alpha}r^{2\alpha} - 2r^{\alpha}\cos(\alpha\pi) + (-i)^{\alpha}}e^{ix\xi}drd\xi\right|\\
&\sim\left|\int_{\left|x\right|^{-1}\leq|\xi|\leq1}|\xi|^{-s}
\left(E_{\alpha,\alpha}((-i)^{\alpha}|\xi|^{\beta})-|\xi|^{\delta-\beta}e^{-i|\xi|^{\delta}}\right)e^{ix\xi}d\xi\right|\\
&\lesssim\left|\frac{1}{x}\int_{|x|^{-1}\leq|\xi|\leq 1}\frac{d}{d\xi}\bigg[|\xi|^{-s}\bigg(\sum_{j=0}\frac{i^{-\alpha j}|\xi|^{\beta j}}{\Gamma(\alpha j+\alpha)}-|\xi|^{\delta-\beta}e^{-i|\xi|^{\delta}}\bigg)\bigg]e^{ix\xi}d\xi\right|\\
&\lesssim\frac{1}{|x|}\left|\int_{|x|^{-1}\leq|\xi|}|\xi|^{-s-1}d\xi\right|\lesssim |x|^{s-1},
\end{align*}
hence we get the estimate \eqref{heat estimate 2}.

Next, by using the Minkovski inequality, the Hardy-Littlewood inequality, we have that
\begin{align*}
&\left\|\int_{0}^{t}\int_{0}^{\infty} \frac{D^{\delta-\beta} e^{-r(t-\tau)D^{\delta}} r^\alpha}{i^{\alpha} r^{2\alpha}
- 2r^\alpha \cos(\alpha \pi) + {(-i)}^{\alpha}} \, drf(\tau,\cdot)d\tau\right\|_{L^{p_{s}}_{x}L^{\infty}_{T}}\\
&\lesssim\int_{0}^{T}\left\|\int_{0}^{\infty} \frac{D^{\delta-\beta} e^{-r(t-\tau)D^{\delta}} r^\alpha}{i^{\alpha} r^{2\alpha}
- 2r^\alpha \cos(\alpha \pi) + {(-i)}^{\alpha}} \, dr\chi_{(0,t)}(\tau)f(\tau,\cdot)\right\|_{L^{p_{s}}_{x}L^{\infty}_{T}}d\tau\\
&\lesssim \int_{0}^{T}(t-\tau)^{\alpha-1}\left\|\left|x\right|^{s-1}\star D^{s}f(\tau,\cdot)\right\|_{L^{p_{s}}_{x}}d\tau\\
&\lesssim T^{\alpha-\frac{1}{2}}\left\|D^{s}f\right\|_{L^{2}_{T,x}},
\end{align*}
hence we get the estimate \eqref{smoothing 19}. The proof is completed.
\end{proof}
Moreover, we can construct the following interpolation estimate.

\begin{lemma}\label{interpolation estimate}
Under the conditions \rm{Lemma \ref{smoothing effects 1}-Lemma \ref{maximum function estimate}} are satisfied, for any $\theta\in\left(0,1\right)$, the following estimates hold:
 \begin{align}
 \label{smoothing 20}
 \left\|e^{-itD^{\delta}}f\right\|_{L^{\frac{p_{s}}{\theta}}_{x}L^{\frac{2}{1-\theta}}_{t}}&\lesssim \left\|D^{\theta s-\left(1-\theta\right)\gamma}f\right\|_{L^{2}_{x}},\\ \label{smoothing 21}
 \left\|\int_{0}^{\infty}
\frac{e^{-rtD^{\delta}}r^{\alpha-1}}
{i^{\alpha}r^{2\alpha} - 2r^{\alpha}\cos(\alpha\pi) + (-i)^{\alpha}} \, drf(x)\right\|_{L^{\frac{p_{s}}{\theta}}_{x}L^{\frac{2}{1-\theta}}_{T}}&\lesssim T^{\frac{(\gamma-\gamma')(1-\theta)}{\delta}}
\left\|D^{\theta s-\left(1-\theta\right)\gamma'}f\right\|_{L^{2}_{x}},\\\label{smoothing 22}
\left\|\int_{0}^{t} D^{\delta-\beta} e^{-i(t-\tau)D^\delta} f(\tau,\cdot) \, d\tau\right\|_{L^{\frac{p_{s}}{\theta}}_{x}L^{\frac{2}{1-\theta}}_{T}}
&\lesssim \left\|D^{\theta\left(\delta-\beta+s\right)-\left(1-\theta\right)\nu}f\right\|_{L^{1}_{T}L^{2}_{x}},\\\label{smoothing 23}\left\|\int_{0}^{t}\int_{0}^{\infty} \frac{D^{\delta-\beta} e^{-r(t-\tau)D^{\delta}} r^\alpha}{i^{\alpha} r^{2\alpha}
- 2r^\alpha \cos(\alpha \pi) + {(-i)}^{\alpha}} \, drf(\tau,\cdot)d\tau\right\|_{L^{\frac{p_{s}}{\theta}}_{x}L^{\frac{2}{1-\theta}}_{T}}&\lesssim T^{\frac{1-2(1-\alpha)\theta}{2}+\frac{(\nu-\nu')(1-\theta)}{\delta}}\left\|D^{\theta s-\left(1-\theta\right)\nu'}f\right\|_{L^{2}_{T,x}}.
\end{align}
\end{lemma}
\begin{proof}
The proof follows the Stein's theorem of analytic interpolation, can refer to \cite{Grafakos}.

Define the operator
\begin{align*}
T_{z}f=D^{-sz+\left(1-z\right)\gamma}e^{-itD^{\delta}}f, \text{ for }z\in\mathbb{C}, 0\leq\Re{z}\leq 1.
\end{align*}
We can get
\begin{align*}
T_{i\omega}f&=D^{\gamma}D^{-i\omega\left(s+\gamma\right)}e^{-itD^{\delta}}f,\\
T_{1+i\omega}f&=D^{-s}D^{-i\omega\left(s+\gamma\right)}e^{-itD^{\delta}}f.
\end{align*}
Note that the interpolation
$$
\left( BMO_{x}L_{t}^{2}, L^{p_{s}}_{x}L^{\infty}_{t} \right)_{\theta} = L^{\frac{p_{s}}{\theta}}_{x}L^{\frac{2}{1-\theta}}_{t},
$$
Combine with the estimate \eqref{BMO estimate}, \eqref{smoothing 1}, \eqref{smoothing 16}, we respectively estimate that
\begin{align*}
\left\|T_{i\omega}f\right\|_{BMO_{x}L^{2}_{t}}&\leq C_{\omega,s,\gamma}\left\|D^{\gamma}e^{-itD^{\delta}}f\right\|_{L^{\infty}_{x}L^{2}_{t}}\lesssim\left\|f\right\|_{L^{2}_{x}},\\
\left\|T_{1+i\omega}f\right\|_{L^{p_{s}}_{x}L^{\infty}_{t}}&\lesssim \left\|D^{-s}e^{-itD^{\delta}}D^{-i\omega\left(s+\gamma\right)}f\right\|_{L^{p_{s}}_{x}L^{\infty}_{t}}\lesssim
\left\|D^{-i\omega\left(s+\gamma\right)}f\right\|_{L^{2}_{x}}\lesssim\left\|f\right\|_{L^{2}_{x}},
\end{align*}
By the Stein's theorem of analytic interpolation and \eqref{Kenig interpolation},
we get that
$$
\left\|e^{-itD^{\delta}}f\right\|_{L^{\frac{p_{s}}{\theta}}_{x}L^{\frac{2}{1-\theta}}_{t}}\lesssim \left\|D^{\theta s-\left(1-\theta\right)\gamma}f\right\|_{L^{2}_{x}},
$$
this implies that \eqref{smoothing 20}.

Define the operator
\begin{align*}
T_{z}f=D^{-sz+\left(1-z\right)\gamma'}\int_{0}^{\infty}\frac{e^{-rtD^{\delta}}r^{\alpha-1}}{i^{\alpha}r^{2\alpha}-2r^{\alpha}
\cos(\alpha\pi)+(-i)^{\alpha}}drf, \text{ for }z\in\mathbb{C}, 0\leq\Re{z}\leq 1, 0<\gamma'<\gamma,
\end{align*}
and by using the estimate \eqref{Kenig interpolation}, \eqref{smoothing 10}, \eqref{smoothing 17} and follows the process of \eqref{smoothing 20}, we obtain
\begin{align*}
\left\|T_{i\omega}f\right\|_{BMO_{x}L^{2}_{T}}\leq C_{\omega,s,\gamma}\bigg\|D^{\gamma'}\int_{0}^{\infty}\frac{e^{-rtD^{\delta}}r^{\alpha-1}}{i^{\alpha}r^{2\alpha}-2r^{\alpha}
\cos(\alpha\pi)+(-i)^{\alpha}}drf\bigg\|_{L^{\infty}_{x}L^{2}_{T}}\lesssim T^{\frac{\gamma-\gamma'}{\delta}}\big\|f\big\|_{L^{2}_{x}},\\
\left\|T_{1+i\omega}f\right\|_{L^{p_{s}}_{x}L^{\infty}_{T}}\lesssim\bigg\|D^{-s}\int_{0}^{\infty}
\frac{e^{-rtD^{\delta}}r^{\alpha-1}}{i^{\alpha}r^{2\alpha}-2r^{\alpha}
\cos(\alpha\pi)+(-i)^{\alpha}}drD^{-i\omega(s+\gamma)}f\bigg\|_{L^{p_{s}}_{x}L^{\infty}_{T}}\lesssim\big\|f\big\|_{L^{2}_{x}},
\end{align*}
combine the interpolation $\left( BMO_{x}L_{T}^{2}, L^{p_{s}}_{x}L^{\infty}_{T} \right)_{\theta} = L^{\frac{p_{s}}{\theta}}_{x}L^{\frac{2}{1-\theta}}_{T}$, we can get \eqref{smoothing 21}.

Moreover, define the operator
\begin{align*}
T_{z}f&=D^{-sz+\left(1-z\right)\nu}_{x}\int_{0}^{t} D^{\delta-\beta} e^{-i(t-\tau)D^\delta} f(\tau,\cdot) \, d\tau, \text{ for }z\in\mathbb{C}, 0\leq\Re{z}\leq 1,\\
T_{z}f&=D^{-sz+\left(1-z\right)\nu'}_{x}\int_{0}^{t}\int_{0}^{\infty} \frac{D^{\delta-\beta} e^{-r(t-\tau)D^{\delta}} r^\alpha}{i^{\alpha} r^{2\alpha}
- 2r^\alpha \cos(\alpha \pi) + {(-i)}^{\alpha}} \, drf(\tau,\cdot)d\tau, \text{ for }z\in\mathbb{C}, 0\leq\Re{z}\leq 1.
\end{align*}
combine with \eqref{BMO estimate}, \eqref{smoothing 9}, \eqref{smoothing 11},\eqref{smoothing 18},\eqref{smoothing 19} and the Stein's theorem of analytic interpolation, similar the process \eqref{smoothing 20} and \eqref{smoothing 21},we can get the estimates \eqref{smoothing 22}, \eqref{smoothing 23}.
\end{proof}
At the end of this section, we generalize the dispersive estimates proven by Ponce \cite{Ponce}, which play a crucial role in the proof of well-posedness for the case $\beta > 2$ in Section 4.
\begin{lemma}\label{diepersive estimate}
Let $d\geq1$ and define for any $\varpi\geq 2$
\begin{align*}
S^{\varpi}(x)=\int_{\mathbb{R}^{d}}e^{i\left(\left|\xi\right|^{\varpi}+x\xi\right)}d\xi.
\end{align*}
Then for any $0\leq\frac{\eta}{d}\leq\frac{\varpi}{2}-1$,
\begin{align}\label{dispersive bounded}
\left\|I^{\eta}S^{\varpi}(x)\right\|_{\infty}\leq C_{\varpi,\eta},
\end{align}
where
\begin{align*}
I^{\eta}S^{\varpi}(x)=\int_{\mathbb{R}^{d}}e^{i\left(\left|\xi\right|^{\varpi}+x\xi\right)}\left|\xi\right|^{\eta}d\xi,
\end{align*}
and $C_{\varpi}$ is only depending on $\varpi$.
\end{lemma}
\begin{remark}
If $\varpi=3$, $d=1$, $S^{\varpi}(x)$ agrees with the Airy function $Ai(x)$.
\end{remark}
\begin{proof}
For $d=1$, the proof can be founded in Ponce \cite{Ponce}, hence we only verify for $d\geq 2$.

The proof depends the Van der Corput Lemma, the properties of the Bessel function and the Fa\'{a} di Bruno's formula.

We choose a $C^{\infty}(\mathbb{R}^{d})$ radial function $\chi_{0}$ such that
$\chi_{0}(\xi)=
\begin{cases}
0,\quad \left|\xi\right|<1,\\
1,\quad \left|\xi\right|\geq 2,
\end{cases}$
hence we only need to verify that
\begin{align*}
H^{\varpi}(x)=\left|\int_{\mathbb{R}^{d}}e^{i\left(\left|\xi\right|^{\varpi}+x\xi\right)}
\chi_{0}(\xi)\left|\xi\right|^{\eta}d\xi\right|\leq C.
\end{align*}
By using the Littlewood-Paley decomposition and using the changes variable, we get that
\begin{align*}
\left\|H^{\varpi}\left(\cdot\right)\right\|_{\infty}&\lesssim\sup_{x\in\mathbb{R}^{d}}\sum_{j\in\mathbb{Z}}
\left|\int_{\mathbb{R}^{d}}e^{i\left(\left|\xi\right|^{\varpi}+x\xi\right)}
\chi_{0}(\xi)\left|\xi\right|^{\eta}\Phi\left(2^{-j}\xi\right)d\xi\right|\\
&\lesssim\sup_{x\in\mathbb{R}^{d}}\sum_{j\in\mathbb{Z}}2^{j\left(\eta+d\right)}
\left|\int_{\mathbb{R}^{d}}e^{i\left(2^{j\varpi}\left|\xi\right|^{\varpi}+x\cdot2^{j}\xi\right)}
\chi_{0}(2^{j}\xi)\left|\xi\right|^{\eta}\Phi\left(\xi\right)d\xi\right|\\
&=\sup_{x\in\mathbb{R}^{d}}\sum_{j=0}2^{j\left(\eta+d\right)}
\left|\int_{\mathbb{R}^{d}}e^{i\left(2^{j\varpi}\left|\xi\right|^{\varpi}+x\cdot2^{j}\xi\right)}
\chi_{0}(2^{j}\xi)\left|\xi\right|^{\eta}\Phi\left(\xi\right)d\xi\right|\\
&=\sup_{x\in\mathbb{R}^{d}}\sum_{j=0}2^{j\left(\eta+d\right)}
\int_{S^{d-1}}\int_{0}^{\infty}e^{i\left(2^{j\varpi}r^{\varpi}+2^{j}r\left(x,\xi'\right)\right)}
\chi_{0}(2^{j}r)r^{\eta+d-1}\Phi\left(r\right)drd\mu(\xi'),
\end{align*}
where $\mu(\cdot)$ is the measure of the surface $S^{d-1}$.

We need to estimate
\begin{align}\label{need estimate 1 }
\sum_{j=0}2^{j\left(\eta+d\right)}
\int_{0}^{\infty}e^{i\left(2^{j\varpi}r^{\varpi}+2^{j}r\left(x,\xi'\right)\right)}
\chi_{0}(2^{j}r)r^{\eta+d-1}\Phi\left(r\right)dr
\end{align}
Note that $r\in\left(\frac{1}{2},2\right)$, and for the function $g(r)=2^{j\varpi}r^{\varpi}+2^{j}r\left(x,\xi'\right)$,
$g'(r)\sim 2^{j\varpi}r^{\varpi-1}+2^{j}\left(x,\xi'\right)$. We can choose constant $c_{1}$, $c_{2}$, and get that
\begin{align*}
&\sum_{j=0}2^{j\left(\eta+d\right)}
\int_{0}^{\infty}e^{i\left(2^{j\varpi}r^{\varpi}+2^{j}r\left(x,\xi'\right)\right)}
\chi_{0}(2^{j}r)r^{\eta+d-1}\Phi\left(r\right)dr\\
&=\sum_{k=1}^{3}\sum_{j\in A_{k}}2^{j\left(\eta+d\right)}
\int_{0}^{\infty}e^{i\left(2^{j\varpi}r^{\varpi}+2^{j}r\left(x,\xi'\right)\right)}
\chi_{0}(2^{j}r)r^{\eta+d-1}\Phi\left(r\right)dr,
\end{align*}
where
\begin{align*}
A_{1}&=\left\{j\geq 0,  c_{1}2^{j\left(\varpi-1\right)}\geq \left|x\right|\right\},\\
A_{2}&=\left\{j\geq 0,  c_{2}2^{j\left(\varpi+1\right)}\leq \left|x\right|\right\},\\
A_{3}&=\left\{j\geq 0,  c_{1}2^{j\left(\varpi-1\right)}\leq \left|x\right|\leq c_{2}2^{j\left(\varpi+1\right)} \right\}.
\end{align*}

We make sequential estimates for $j \in A_i$ ($i=1,2,3$).

For $j\in A_{1}$, denote the operator $L$:
$$
L=\frac{1}{ig'(r)}\frac{d}{dr},\text{ then we have }L\left(e^{ig(r)}\right)=e^{ig(r)},
$$

We denote that $S_{j}(r)=\chi_{0}(2^{j}r)r^{\eta+d-1}\Phi\left(r\right)$, obviously, $S_{j}(r)=0$ for $j\geq 2$. We derive that
$$
\left|\partial_{r}^{k}S_{j}(r)\right|\lesssim 1, \text{ for any }k\in\mathbb{Z}^{+}.
$$
Moreover, note that $\left|x\right|\lesssim 2^{j\left(\varpi-1\right)}$, we derive that
$$
\left|g'(r)\right|\gtrsim 2^{j\varpi}\text{ and }\left|g^{(k)}(r)\right|\lesssim 2^{j\varpi}\text{ for any }k\in\mathbb{Z}^{+}\text{ and }k\geq 2.
$$
By using the Fa\'{a} di Bruno's formula, we get that
\begin{align*}
\left|D^{k}\frac{1}{g'(r)}\right|\lesssim \left|\sum_{\substack {b_{1}+b_{2}+\cdot\cdot\cdot+b_{l}=k\\ b_{j}\geq 0,1\leq l\leq k}}
\prod_{j=1}^{l}\left(\partial^{b_{j}}g'(r)\right)\frac{1}{\left(g'(r)\right)^{l+1}}\right|\lesssim 2^{-j\varpi},
\end{align*}
hence we get that for any $N\in\mathbb{Z}^{+}$,
\begin{align*}
\left|\int_{0}^{\infty}e^{i\left(2^{j\varpi}r^{\varpi}+2^{j}r\left(x,\xi'\right)\right)}
\chi_{0}(2^{j}r)r^{\eta+d-1}\Phi\left(r\right)dr\right|&=\left|\int_{0}^{\infty}L^{N}\left(e^{ig(r)}\right)S_{j}(r)dr\right|\\
&=\left|\int_{0}^{\infty}e^{ig(r)}\left(L^{*}\right)^{N}S_{j}(r)dr\right|,
\end{align*}
where
\begin{align*}
\left(L^{*}\right)^{N}S_{j}(r)&=\frac{1}{-i}\left(L^{*}\right)^{N-1}\left\{\frac{d}{dr}\left(\frac{1}{g'(r)}S_{j}(r)\right)\right\}\\
&=\left(\frac{1}{-i}\right)^{2}\left(L^{*}\right)^{N-2}\left\{\frac{d}{dr}\left(\frac{1}{g'(r)}\frac{d}{dr}
\left(\frac{1}{g'(r)}S_{j}(r)\right)\right)\right\}\\
&=\cdot\cdot\cdot\\
&\sim\sum_{_{\alpha+\beta=N}}\sum_{\substack{b_{1}+b_{2}+\cdot\cdot\cdot+b_{l}=\alpha\\b_{j}\geq0,1\leq l\leq N}}\prod_{j=1}^{N}
\left(\frac{d^{b_{j}}}{dr^{b_{j}}}\left(\frac{1}{g'(r)}\right)\right)\frac{d^{\beta}}{dr^{\beta}}S_{j}(r),
\end{align*}
hence we get
\begin{align*}
\left|\int_{0}^{\infty}e^{ig(r)}\left(L^{*}\right)^{N}S_{j}(r)dr\right|\lesssim 2^{-j\varpi N}, \text{ for any }N\in\mathbb{Z}^{+}.
\end{align*}

For $j\in A_{2}$, similar to the $j\in A_{1}$, we can get the same estiamtes, hence we omit the details.

For $ j \in A_{3} $, the above method of integration by parts is no longer applicable, because there exists $ x_{0} \in \left( c_{1} 2^{j (\varpi - 1)}, c_{2} 2^{j (\varpi + 1)} \right) $ such that $ g'(r) = 0 $. Therefore, we need to employ Van der Corput Lemma.

By the Fourier transform formula on $ S^{d-1} $, as referenced in \cite{Grafakos}, and by using the properties of the Bessel function \eqref{Bessel function transform}, \eqref{h function property}, we obtain that
\begin{align*}
&\int_{S^{d-1}}\int_{0}^{\infty}e^{i\left(2^{j\varpi}r^{\varpi}+2^{j}r\left(x,\xi'\right)\right)}
\chi_{0}(2^{j}r)r^{\eta+d-1}\Phi\left(r\right)drd\mu(\xi')\\
&=\int_{S^{d-1}}\int_{0}^{\infty}e^{ig(r)}
\chi_{0}(2^{j}r)r^{\eta+d-1}\Phi\left(r\right)drd\mu(\xi')\\
&=c_{d}\int_{0}^{\infty}\chi_{0}(2^{j}r)r^{\eta+d-1}\Phi\left(r\right)e^{i2^{j\varpi}r^{\varpi}}\left(2^{j}rs\right)^{-\frac{d-2}{2}}
J_{\frac{d-2}{2}}\left(2^{j}rs\right)dr\quad \left(s=\left|x\right|\right)\\
&=c_{d}\int_{0}^{\infty}\chi_{0}(2^{j}r)r^{\eta+d-1}\Phi\left(r\right)e^{i2^{j\varpi}r^{\varpi}}\Re \left(e^{i2^{j}rs}h\left(2^{j}rs\right)\right)dr\\
&\sim c_{d}\int_{0}^{\infty}\chi_{0}(2^{j}r)r^{\eta+d-1}\Phi\left(r\right)e^{i2^{j\varpi}r^{\varpi}+i2^{j}rs}h\left(2^{j}rs\right)dr\\
&\qquad+c_{d}\int_{0}^{\infty}\chi_{0}(2^{j}r)r^{\eta+d-1}\Phi
\left(r\right)e^{i2^{j\varpi}r^{\varpi}-i2^{j}rs}\overline{h}\left(2^{j}rs\right)dr.
\end{align*}
Note $r\in\left(\frac{1}{2},2\right)$, $s\in\left(c_{1}2^{j\left(\varpi-1\right)},c_{2}2^{j\left(\varpi+1\right)}\right)$ and using the lemma \ref{Van der Corput lemma} and equation \eqref{h function property}, we get
\begin{align*}
&\int_{S^{d-1}}\int_{0}^{\infty}e^{i\left(2^{j\varpi}r^{\varpi}+2^{j}r\left(x,\xi'\right)\right)}
\chi_{0}(2^{j}r)r^{\eta+d-1}\Phi\left(r\right)drd\mu(\xi')\\
&\lesssim \left(2^{j\varpi}\right)^{-\frac{1}{2}}\left\{\left\|S_{j}(r)h\left(2^{j}rs\right)\right\|_{\infty}+
\left\|\partial\left(S_{j}(r)h\left(2^{j}rs\right)\right)\right\|_{L^{1}}\right\}\\
&\lesssim  \left(2^{j\varpi}\right)^{-\frac{1}{2}}\left(2^{j}\right)^{-\frac{d-1}{2}}\left(2^{j(\varpi-1)}\right)^{-\frac{d-1}{2}}\\
&\lesssim 2^{\frac{-j\varpi d}{2}}.
\end{align*}
If $j\in A_{3}$, we can get
$$
\frac{1}{\varpi+1}\log_{2}\left(\frac{\left|x\right|}{c_{2}}\right)\leq j\leq\frac{1}{\varpi-1}\log_{2}\left(\frac{\left|x\right|}{c_{1}}\right),
$$
hence we can choose the constant $M$ which is only dependent of $\varpi$, $c_{1}$ and $c_{2}$ such that $\left|A_{3}\right|\leq M$.
From above discuss, we can choose the constan $N\in\mathbb{Z}^{+}$ large enough and for $0\leq\frac{\eta}{d}\leq\frac{\varpi}{2}-1$,
we get
\begin{align*}
&\sum_{j=0}2^{j\left(\eta+d\right)}
\int_{S^{d-1}}\int_{0}^{\infty}e^{i\left(2^{j\varpi}r^{\varpi}+2^{j}r\left(x,\xi'\right)\right)}
\chi_{0}(2^{j}r)r^{\eta+d-1}\Phi\left(r\right)drd\mu(\xi')\\
&\lesssim \sum_{j\in A_{1}}2^{j\left(\eta+d-\varpi N\right)}+\sum_{j\in A_{2}}2^{j\left(\eta+d-\varpi N\right)}+\sum_{j\in A_{3}}2^{j\left(\eta+d-\frac{\varpi d}{2}\right)}\\
&\lesssim 1, \text{ for }0\leq\frac{\eta}{d}\leq\frac{\varpi}{2}-1.
\end{align*}
The proof is completed.
\end{proof}
\begin{remark}\label{scaling disoersive estimates}
When the condition of Lemma \ref{diepersive estimate} is satisfied, by the scaling transform, we have
\begin{align}\label{dispersive estimate with t}
\left|\int_{\mathbb{R}^{d}}e^{i\left(t\left|\xi\right|^{\varpi}+x\xi\right)}\left|\xi\right|^{\eta}d\xi\right|\lesssim
t^{-\frac{d+\eta}{\varpi}}\left\|I^{\eta}S\left(t^{-\frac{1}{\varpi}}x\right)\right\|_{\infty}\lesssim t^{-\frac{d+\eta}{\varpi}}.
\end{align}
\end{remark}
\begin{remark}\label{operator holder estimate}
From Lemma {\rm\ref{diepersive estimate}} and Remark {\rm\ref{scaling disoersive estimates}}, it is easy to see
\begin{align}\label{holder estimate}
\big\|I^{\eta}S^{\varpi}(x)\big\|_{\dot{C}^{\varrho}(\mathbb{R}^{d})}\leq C_{\varpi},\quad \text{when } 0\leq \frac{\eta+\varrho}{d}\leq \frac{\varpi}{2}-1.
\end{align}
\begin{align}\label{holder estimate scailing}
\bigg\|\int_{\mathbb{R}^{d}}e^{i(t|\xi|^{\varpi}+x\xi)}|\xi|^{\eta}\,d\xi\bigg\|_{\dot{C}^{\varrho}}\lesssim t^{-\frac{d+\eta+\varpi}{\varpi}},\quad \text{when } 0\leq \frac{\eta+\varrho}{d}\leq \frac{\varpi}{2}-1.
\end{align}
Indeed, note that $\dot{C}^{\varrho}\sim \dot{B}^{\varrho}_{\infty,\infty}$ and $l^{1}\subset l^{\infty}$, hence
\begin{align*}
\big\|I^{\eta}S^{\varpi}(x)\big\|_{\dot{C}^{\varrho}}&\sim\sup_{j\in\mathbb{Z}}2^{j\varrho}
\big\|\triangle_{j}I^{\eta}S^{\varpi}\big\|_{L^{\infty}}\\
&\lesssim\sup_{x\in\mathbb{R}^{d}}\sum_{j\in\mathbb{Z}}
\left|\int_{\mathbb{R}^{d}}e^{i\left(|\xi|^{\varpi}+x\xi\right)}
|\xi|^{\eta+\varrho}\Phi\left(2^{-j}\xi\right)d\xi\right|,
\end{align*}
which implies \eqref{holder estimate} holds. By scaling, \eqref{holder estimate scailing} also holds.
\end{remark}
By using the dispersive estimate Lemma \ref{diepersive estimate}, we can get the following lemma.
\begin{lemma}\label{estimate operator}
Let $d\geq 1$, and the condition $0\leq\frac{\delta-\beta}{d}<\frac{\delta}{2}-1$ is satisfied, the following estimates hold:
\begin{align}
\label{dispersive estimates 1}\left\|e^{-itD^{\delta}}f\right\|_{L^{p}}&\lesssim t^{-\frac{\alpha d}{\beta}\left(1-\frac{2}{p}\right)}\left\|f\right\|_{L^{p'}},\text{ }2\leq p\leq\infty,\\\label{dispersive estimates 2}
\left\|D^{\delta-\beta}e^{-itD^{\delta}}f\right\|_{L^{p}}&\lesssim t^{\alpha-1-\frac{\alpha d}{\beta}\left(1-\frac{2}{p}\right)}\left\|f\right\|_{L^{p'}},\text{ }\frac{1}{p}=\frac{1}{2}-\frac{\delta-\beta}{d\left(\delta-2\right)},\\\label{dispersive estimates 3}
\left\|\int_{0}^{\infty}
\frac{e^{-rtD^{\delta}}r^{\alpha-1}}
{i^{\alpha}r^{2\alpha} - 2r^{\alpha}\cos(\alpha\pi) + (-i)^{\alpha}} \, drf(\cdot)\right\|_{L^{r}}&\lesssim t^{-\frac{\alpha d}{\beta}\left(\frac{1}{p}-\frac{1}{r}\right)}\left\|f\right\|_{L^{p}},\text{ }\frac{d}{\beta}\big(\frac{1}{p}-\frac{1}{r}\big)<1,\\\label{dispersive estimates 4}
\left\|\int_{0}^{\infty} \frac{D^{\delta-\beta} e^{-rtD^{\delta}} r^\alpha}{i^{\alpha} r^{2\alpha}
- 2r^\alpha \cos(\alpha \pi) + {(-i)}^{\alpha}} \, drf(\cdot)\right\|_{L^{r}}&\lesssim t^{\alpha-1-\frac{\alpha d}{\beta}\left(\frac{1}{p}-\frac{1}{r}\right)}\left\|f\right\|_{L^{p}},\text{ }\frac{d}{\beta}\big(\frac{1}{p}-\frac{1}{r}\big)<2.
\end{align}
\end{lemma}
\begin{proof}
From Lemma \ref{diepersive estimate} and Plancherel identity, we obtain
\begin{align*}
\left\|e^{-itD^{\delta}}f\right\|_{L^{\infty}}&\lesssim t^{-\frac{\alpha d}{\beta}}\left\|f\right\|_{L^{1}},\\
\left\|e^{-itD^{\delta}}f\right\|_{L^{2}}&\lesssim \left\|f\right\|_{L^{2}},
\end{align*}
by using the Riesz interpolation, we get \eqref{dispersive estimates 1}.

We choose $\sigma=\left(\frac{\delta}{2}-1\right)d-\left(\delta-\beta\right)>0$, and note that
$
\frac{\delta-\beta+\sigma+d}{\delta}=\frac{d}{2}
$.
By using Lemma \ref{diepersive estimate} and Plancherel identity,
we have that
\begin{align*}
\left\|D^{\delta-\beta}e^{-itD^{\delta}}f\right\|_{L^{\infty}}
=\left\|D^{\delta-\beta}D^{\sigma}e^{-itD^{\delta}}D^{-\sigma}f\right\|_{L^{\infty}}&\lesssim t^{-\frac{d}{2}}\left\|f\right\|_{\dot{H}^{-\sigma}_{1}},\\
\left\|D^{\delta-\beta}e^{-itD^{\delta}}f\right\|_{L^{2}}=\left\|e^{-itD^{\delta}}D^{\delta-\beta}f\right\|_{L^{2}}
&\lesssim\left\|f\right\|_{\dot{H}^{\delta-\beta}_{2}},
\end{align*}
for
\begin{align*}
\theta=\frac{2\left(\delta-\beta\right)}{d\left(\delta-2\right)},\text{ }
\frac{1}{p}=\frac{1}{2}-\frac{\delta-\beta}{d\left(\delta-2\right)},\text{ }
\frac{1}{p'}=\frac{1}{2}+\frac{\delta-\beta}{d\left(\delta-2\right)},
\end{align*}
and note that
\begin{align*}
\theta\left(-\sigma\right)+\left(1-\theta\right)\left(\delta-\beta\right)&=0,\\
-\frac{d}{2}\theta=\alpha-1-\frac{\alpha d}{\beta}\left(1-\frac{2}{p}\right)&=-\frac{d}{\delta}\big(\frac{1}{p'}-\frac{1}{p}+\frac{\delta-\beta}{d}\big),
\end{align*}
by using the real interpolation and Besov embedeeding \cite{Bergh,B.X. Wang}, we have
\begin{align*}
\left(L^{\infty}(\mathbb{R}^{d}),L^{2}(\mathbb{R}^{d})\right)_{\theta,p}=L^{p,p}(\mathbb{R}^{d})
&=L^{p}(\mathbb{R}^{d}),\\
\left(\dot{H}^{-\sigma}_{1}(\mathbb{R}^{d}),
\dot{H}^{\delta-\beta}_{2}(\mathbb{R}^{d})\right)_{\theta,p}&=\dot{B}^{0}_{p',p}(\mathbb{R}^{d}),\\
L^{p'}(\mathbb{R}^{d})\approx\dot{F}^{0}_{p',2}(\mathbb{R}^{d})\hookrightarrow \dot{B}^{0}_{p',2}(\mathbb{R}^{d})&\hookrightarrow\dot{B}^{0}_{p',p}(\mathbb{R}^{d}),
\end{align*}
hence we obtain that
\begin{align*}
\left\|D^{\delta-\beta}e^{-itD^{\delta}}f\right\|_{L^{p}}&\lesssim t^{\alpha-1-\frac{\alpha d}{\beta}\left(1-\frac{2}{p}\right)}\left\|f\right\|_{L^{p'}},
\end{align*}
hence we obtain \eqref{dispersive estimates 2}.

Moreover, $\frac{d}{\beta}\big(\frac{1}{p}-\frac{1}{r}\big)<1$ and $\frac{d}{\beta}\big(\frac{1}{p}-\frac{1}{r}\big)<2$ can ensure that integral
\begin{align*}
\int_{0}^{\infty}\frac{r^{\alpha-1-\frac{\alpha d}{\beta}\left(\frac{1}{p}-\frac{1}{r}\right)}}{i^{\alpha}r^{2\alpha} - 2r^{\alpha}\cos(\alpha\pi) + (-i)^{\alpha}}dr,\int_{0}^{\infty}\frac{r^{2\alpha-1-\frac{\alpha d}{\beta}\left(\frac{1}{p}-\frac{1}{r}\right)}}{i^{\alpha}r^{2\alpha} - 2r^{\alpha}\cos(\alpha\pi) + (-i)^{\alpha}}dr
\end{align*}
are respectively both convergent. Therefore, basic the estimates of fractional heat kernel, we obtain \eqref{dispersive estimates 3},\eqref{dispersive estimates 4}.
The proof is completed.
\end{proof}
\begin{remark}\label{condition claim}
From the above process, it is obvious that the condition $0\leq \frac{\delta-\beta}{d}<\frac{\delta}{2}-1$ is not necessary for \eqref{dispersive estimates 1}, \eqref{dispersive estimates 3} and \eqref{dispersive estimates 4}.
\end{remark}
\begin{remark}
Indeed, let $r\in(0,1)$. If $0\leq (\delta-\beta)/d\leq\delta/2-1/r$, then for the operator $D^{\delta-\beta}e^{-itD^{\delta}}$, we obtain the following dispersive estimates:
\begin{enumerate}[\rm(i)]
\item
\begin{align}\label{Hardy estimate1}
\big\|D^{\delta-\beta}e^{-itD^{\delta}}f\big\|_{L^{\infty}}&\lesssim t^{\alpha-1-\frac{\alpha d}{\beta r}}\big\|f\big\|_{\mathcal{H}^{r}},
\end{align}
\item
Consider the line segment $\overline{AB}$ determined by points $A:(\frac{1}{p'},\frac{1}{p})$ and $B:(\frac{1}{r},0)$, where $p$ takes values from \eqref{dispersive estimates 2}. Then we have
\begin{align}\label{Hardy estimate2}
\big\|D^{\delta-\beta}e^{-itD^{\delta}}f\big\|_{\mathcal{H}^{a}}&\lesssim t^{\alpha-1-\frac{\alpha d}{\beta}(\frac{1}{b}-\frac{1}{a})}\big\|f\big\|_{\mathcal{H}^{b}},\ \text{where } \left(\frac{1}{b},\frac{1}{a}\right)\in \overline{AB}.
\end{align}
\item
\begin{align}\label{Hardy estimate3}
\big\|D^{\delta-\beta}e^{-itD^{\delta}}f\big\|_{BMO}\lesssim t^{\alpha-1-\frac{\alpha d}{\beta c}}\big\|f\big\|_{\mathcal{H}^{c}},\ \text{where } c=\frac{p+r-pr}{p-pr+2r-1}.
\end{align}
\end{enumerate}

For $r\in(0,1)$, note that $\big(\mathcal{H}^{r}(\mathbb{R}^{d})\big)^{\star}=\dot{C}^{d(\frac{1}{r}-1)}(\mathbb{R}^{d})$, and
$$
0\leq \frac{\delta-\beta}{d}<\frac{\delta}{2}-\frac{1}{r}\Rightarrow \frac{\delta-\beta+d(\frac{1}{r}-1)}{d}\leq\frac{\delta}{2}-1.
$$
Applying Lemma \ref{diepersive estimate} and Remark \ref{operator holder estimate}, we have
\begin{align*}
\big\|D^{\delta-\beta}e^{-itD^{\delta}}f\big\|_{L^{\infty}}&=\big\|S_{t,\delta-\beta}\star f\big\|_{L^{\infty}}\\
&\lesssim\big\|\mathcal{F}^{-1}\big(|\xi|^{\delta-\beta}e^{-it|\xi|^{\delta}}\big)\big\|_{\dot{C}^{d(\frac{1}{r}-1)}}\big\|f\big\|_{\mathcal{H}^{r}}\\
&\lesssim t^{-\frac{d}{\delta}\left(\frac{1}{r}+\frac{\delta-\beta}{d}\right)}\big\|f\big\|_{\mathcal{H}^{r}}.
\end{align*}
Combining \eqref{dispersive estimates 2} and \eqref{Hardy estimate1}, interpolation yields \eqref{Hardy estimate2}. Furthermore, since $(1,\frac{1-r}{p+r-pr})\in \overline{AB}$, we obtain
$$
\big\|D^{\delta-\beta}e^{-itD^{\delta}}f\big\|_{\mathcal{H}^{\frac{p+r-pr}{1-r}}}\lesssim t^{-\frac{\alpha d}{\beta}\left(\frac{p-pr+2r-1}{p+r-pr}+\frac{\delta-\beta}{d}\right)}\big\|f\big\|_{\mathcal{H}^{1}}.
$$
Noting that $\big(\mathcal{H}^{1}(\mathbb{R}^{d})\big)^{\star}=BMO$, we derive \eqref{Hardy estimate3} through duality estimates.
\end{remark}
Basic the above Lemma \ref{estimate operator}, we can construct the following estimates for the solution operator $E_{\alpha,1}\left((-it)^{\alpha}D^{\beta}\right)$ and $E_{\alpha,\alpha}\left((-it)^{\alpha}D^{\beta}\right)$. Under the conditions of Lemma \ref{estimate operator}, for convenience, we denote
$p_{0}=\frac{2d(\delta-2)}{d(\delta-2)-2(\delta-\beta)}$.
\begin{lemma}\label{Mittag-Leffer operator estimate}
Let $d\geq 1$, $\beta>2$, under the conditions of Lemma {\rm\ref{estimate operator}},
the following estimate hold:
\begin{align}
\label{Mittag-Leffer estimates 1}
\left\|E_{\alpha,1}\left((-it)^{\alpha}D^{\beta}\right)f(\cdot)\right\|_{L^{p_{0}}}
\lesssim t^{-\frac{\alpha d}{\beta}\left(1-\frac{2}{p}\right)}\left\|f\right\|_{L^{p_{0}'}},\\\label{Mittag-Leffer estimates 2}
\left\|E_{\alpha,\alpha}\left((-it)^{\alpha}D^{\beta}\right)f(\cdot)\right\|_{L^{p_{0}}}
\lesssim t^{-\frac{\alpha d}{\beta}\left(1-\frac{2}{p}\right)}\left\|f\right\|_{L^{p_{0}'}}.
\end{align}
\end{lemma}
\begin{proof}
From Remark \ref{change variable equaton of Mittag-Leffer function}, we obtain that
\begin{align*}
E_{\alpha,1}\left((-it)^{\alpha}D^{\beta}\right)&\sim e^{-it D^{\delta}}
    -\int_{0}^{\infty}
    \frac{e^{-rtD^{\delta}} r^{\alpha-1}}
    {i^{\alpha} r^{2\alpha} - 2 r^{\alpha} \cos(\alpha \pi) + (-i)^{\alpha}} \, dr,\\
E_{\alpha,\alpha}\left((-it)^{\alpha}D^{\beta}\right)&\sim t^{1-\alpha} D^{\delta-\beta} e^{-itD^{\delta}}
    - t^{1-\alpha} \int_{0}^{\infty}
    \frac{D^{\delta-\beta}e^{-rtD^{\delta}}r^{\alpha}}
    {i^{\alpha} r^{2\alpha} - 2 r^{\alpha} \cos(\alpha \pi) + (-i)^{\alpha}} \, dr.
\end{align*}
The condition $\beta>2$ can ensure that
$$
\frac{1}{p_{0}}=\frac{1}{2}-\frac{\delta-\beta}{d\left(\delta-2\right)}>\frac{d-\beta}{2d},
$$
by using Lemma \ref{estimate operator}, the estimates \eqref{Mittag-Leffer estimates 1} and \eqref{Mittag-Leffer estimates 2}
are obvious. The proof is now completed.
\end{proof}
\begin{remark}
From Lemma \ref{estimate operator}, it is obvious that \eqref{Mittag-Leffer estimates 1} hold for any $\frac{d-\beta}{2d}<\frac{1}{p}\leq\frac{1}{2}$.
\end{remark}

\section{Main results}
In this section, we establish our main results, which include global/local well-posedness, the asymptotic behavior of solutions, and self-similarity. For convenience, we choose the nonlinearity $g(w) = |w|^{p-2}w$.
\subsection{The well-posedness}

\subsubsection{The case $2<\beta<\infty$}
First, we construct the global/local well-posedness of E.q. \eqref{Eq:F-T-S-E} for case $2<\beta<\infty$.

For $2\leq p<\infty$, we consider the following Banach space
$$
X^{\kappa}_{p}=\left\{w:\left(0,\infty\right)\rightarrow L^{p}(\mathbb{R}^{d})/\quad \sup_{0<t<\infty}t^{\kappa}\left\|w(t)\right\|_{p}<\infty\right\}\text{ with }\kappa=\frac{\alpha}{\beta}\left(\frac{\beta}{p-2}-\frac{d}{p}\right),
$$
and define the initial value space $X_{0}$ as the set of $u_{0}\in\mathcal{S}'$ such that
$$
\left\|w_{0}\right\|_{X_{0}}=\sup_{0<t<\infty}t^{\kappa}\left\|E_{\alpha,1}((-it)^{\alpha}
D^{\beta})w_{0}\right\|_{L^{p}}<\infty,
$$
then we can construct the following result.
\begin{theorem}\label{global existence}
Let $\beta>2$, $d\geq 1$ and for $\frac{\delta-\beta}{d}<\frac{\delta}{2}-1$, assume that $g(w)=|w|^{p_{0}-2}w$ and the condition $\kappa\left(p_{0}-1\right)<1$ is satisfied. For the constant $\varepsilon>0$ is small enough, the E.q. \eqref{Eq:F-T-S-E} has a unique global mild solution in $X_{p_{0}}^{\kappa}$ when the initial $\left\|w_{0}\right\|_{X_{0}}\lesssim \varepsilon$. If $u,w$ are respectively the two mild solutions of E.q. \eqref{Eq:F-T-S-E} with the initial data $w_{0}$ and $u_{0}$, we have that $\left\|w-u\right\|_{X_{p_{0}}^{\kappa}}\lesssim \left\|w_{0}-u_{0}\right\|_{X_{0}}$.
\end{theorem}
\begin{remark}
The assumption condition in Theorem \ref{global existence}, $\kappa(p_{0}-1) < 1$, is reasonable. For example, let $p^{*}$ be the positive real root of the equation
\[
dx^2 - (3d + \beta)x + 2d = 0.
\]
For $d = 3$, $\beta \in \left(2, \frac{9}{2}\right)$, and $\alpha \in \left(0, \frac{3}{8}\right)$, by calculation, we have
\[
\frac{d-\beta}{2d}<\frac{1}{p_{0}}<\frac{1}{p^{*}},
\]
which ensures that the condition $\kappa(p_{0}-1) < 1$ is satisfied.
\end{remark}
\begin{proof}
Consider the Banach space
\[
\Xi_{\varepsilon} = \left\{ w \in X_{p_{0}}^{\kappa} : \quad \left\| w \right\|_{X_{p_{0}}^{\kappa}} \lesssim \varepsilon \right\},
\]
equipped with the norm metric
\[
\left\| w - v \right\|_{X_{p_{0}}^{\kappa}} = \sup_{0 < t < \infty} t^{\kappa} \left\| w(t) - v(t) \right\|_{p_{0}}.
\]
Define the operator
$$
\Psi(w)=E_{\alpha,1}((-it)^{\alpha}D^{\beta})w_{0}(x) + i^{-\alpha}
\int_{0}^{t}(t-s)^{\alpha-1}E_{\alpha,\alpha}((-i(t-s))^{\alpha}D^{\beta})g(w(s,x)) \, ds,
$$
where $g\left(w\right)=\left|w\right|^{p_{0}-2}w$.
From the Lemma \ref{Mittag-Leffer operator estimate}, we obtain that
\begin{align*}
\left\|\Psi(w)\right\|_{p_{0}}&\lesssim \left\|E_{\alpha,1}((-it)^{\alpha}D^{\beta})w_{0}\right\|_{p_{0}}
+\int_{0}^{t}(t-s)^{\alpha-1}\left\|E_{\alpha,\alpha}((-i(t-s))^{\alpha}D^{\beta})g(w(s,\cdot))\right\|_{p_{0}}ds\\
&\lesssim \left\|E_{\alpha,1}((-it)^{\alpha}D^{\beta})w_{0}\right\|_{p_{0}}+\int_{0}^{t}\left(t-s\right)^{\alpha-1-\frac{\alpha d}{\beta}\left(\frac{1}{p_{0}'}-\frac{1}{p_{0}}\right)}\left\|g(w(s,\cdot))\right\|_{p_{0}'}ds\\
&\lesssim \left\|E_{\alpha,1}((-it)^{\alpha}D^{\beta})w_{0}\right\|_{p_{0}}+\int_{0}^{t}\left(t-s\right)^{\alpha-1-\frac{\alpha d}{\beta}\left(\frac{1}{p_{0}'}-\frac{1}{p_{0}}\right)}\left\|w\left(s,\cdot\right)\right\|^{p_{0}-1}_{p_{0}}ds\\
&\lesssim \left\|E_{\alpha,1}((-it)^{\alpha}D^{\beta})w_{0}\right\|_{p_{0}}+\int_{0}^{t}\left(t-s\right)^{\alpha-1-\frac{\alpha d}{\beta}\left(\frac{1}{p_{0}'}-\frac{1}{p_{0}}\right)}s^{-\kappa\left(p_{0}-1\right)}ds\left\|w\right\|_{X_{p_{0}}^{\kappa}}^{p_{0}-1}\\
&\lesssim \left\|E_{\alpha,1}((-it)^{\alpha}D^{\beta})w_{0}\right\|_{p_{0}}+t^{-\kappa}B_{\kappa}\left\|w\right\|_{X_{p_{0}}^{\kappa}}^{p_{0}-1},
\end{align*}
where the constant
$$
B_{\kappa}=B\left(\alpha-\frac{\alpha d}{\beta}(\frac{1}{p_{0}'}-\frac{1}{p_{0}}),1-\kappa(p_{0}-1)\right).
$$
Therefore, we obtain that $\left\|\Psi(w)\right\|_{X_{p_{0}}^{\kappa}}\lesssim \varepsilon$ when the initial $\left\|w_{0}\right\|_{X_{0}}\lesssim \varepsilon$ is small enough, and the operator $\Psi$ maps $\Xi_{\varepsilon}$ to $\Xi_{\varepsilon}$.

Moreover, for any $w,v\in\Xi_{\varepsilon}$ and small $\varepsilon>0$, we obtain that
\begin{align*}
&t^{\kappa}\left\|\Psi(w)-\Psi(v)\right\|_{p_{0}}\\
&\lesssim t^{\kappa}\int_{0}^{t}\left(t-s\right)^{\alpha-1-\frac{\alpha d}{\beta}(\frac{1}{p_{0}'}-\frac{1}{p_{0}})}\left\|g(w(s))-g(v(s))\right\|_{p_{0}'}ds\\
&\lesssim t^{\kappa}\int_{0}^{t}\left(t-s\right)^{\alpha-1-\frac{\alpha d}{\beta}(\frac{1}{p_{0}'}-\frac{1}{p_{0}})}\left(\left\|w(s)\right\|^{p_{0}-2}_{p_{0}}+\left\|v(s)\right\|^{p_{0}-2}_{p_{0}}\right)
\left\|w(s)-v(s)\right\|_{p}ds\\
&\lesssim t^{\kappa}\int_{0}^{t}\left(t-s\right)^{\alpha-1-\frac{\alpha d}{\beta}(\frac{1}{p_{0}'}-\frac{1}{p_{0}})}s^{-\kappa(p_{0}-1)}
\left((s^{\kappa}\left\|w(s)\right\|_{p_{0}})^{p_{0}-2}+(s^{\kappa}\left\|v(s)\right\|_{p_{0}})^{p_{0}-2}\right)
(s^{\kappa}\left\|w(s)-v(s)\right\|_{p_{0}})ds\\
&\lesssim 2^{p_{0}-2}B_{\kappa}\varepsilon^{p_{0}-2}\left\|w-v\right\|_{X_{p_{0}}^{\kappa}}<\frac{1}{2}\left\|w-v\right\|_{X_{p_{0}}^{\kappa}}.
\end{align*}
We get that
$$
\left\|\Psi(w)-\Psi(v)\right\|_{X_{p_{0}}^{\kappa}}<\frac{1}{2}\left\|w-v\right\|_{X_{p_{0}}^{\kappa}}.
$$
The operator $\Psi$ exist the unique fixed point on $\Xi_{\kappa}$ by Banach contraction mapping theorem, which is the mild solution of E.q. \eqref{Eq:F-T-S-E}.

Therefore, for two solutions $w,v$ with the initial data $w_{0}$ and $v_{0}$, respectively, similar the above process, we obtain that
\begin{align*}
\left\|w-v\right\|_{X_{p_{0}}^{\kappa}}&\lesssim \left\|w_{0}-v_{0}\right\|_{X_{0}}+2^{p_{0}-2}B_{\kappa}\varepsilon^{p_{0}-2}\left\|w-v\right\|_{X_{p_{0}}^{\kappa}}\\
&\lesssim\left\|w_{0}-v_{0}\right\|_{X_{0}}+\frac{1}{2}\left\|w-v\right\|_{X_{p_{0}}^{\kappa}},
\end{align*}
this implies that $\left\|w-v\right\|_{X_{0}^{\kappa}}\lesssim \left\|w_{0}-v_{0}\right\|_{X_{0}}$.
The proof is now completed.
\end{proof}
Next we construct the local existence of E.q. \eqref{Eq:F-T-S-E} for case $2<\beta<\infty$.

For $2\leq p<\infty$, $T>0$, we consider the following Banach space
$$
Y^{\nu}_{p}=\left\{w:\left(0,T\right)\rightarrow L^{p}(\mathbb{R}^{d})/\quad \sup_{0<t<T}t^{\nu}\left\|w(t)\right\|_{p}<\infty\right\}\text{ with }\nu=\frac{\alpha d}{\beta}\left(\frac{1}{p'}-\frac{1}{p}\right),
$$
then we have the following result.
\begin{theorem}\label{local existence}
Let $\beta>2$, $d\geq 1$ and for $\frac{\delta-\beta}{d}<\frac{\delta}{2}-1$, assume that $g(w)=|w|^{p_{0}-2}w$ and the condition $\nu\left(p_{0}-1\right)<\alpha$ is satisfied. For any $w_{0}\in L^{p_{0}'}(\mathbb{R}^{d})$, there exist $T>0$ such that E.q. \eqref{Eq:F-T-S-E} has a unique local mild solution in $Y_{p_{0}}^{\nu}$. If $w,v$ are respectively the two local mild solutions of E.q. \eqref{Eq:F-T-S-E} with the initial data $w_{0}$ and $v_{0}$, we have that $\left\|w-v\right\|_{Y_{p_{0}}^{\nu}}\lesssim \left\|w_{0}-v_{0}\right\|_{p_{0}'}$.
\end{theorem}
\begin{remark}
The assumption condition in Theorem \ref{local existence}, $\nu(p_{0}-1) < \alpha$, is reasonable. For example, let $p^{*}$ be the positive real root of the equation
\[
dx^2 - (3d + \beta)x + 2d = 0.
\]
For $d = 3$, $\beta>\frac{9}{2}$, and $\alpha\in\left(\frac{3}{8},1\right)$, by calculation, we have
\[
\frac{d-\beta}{2d}<\frac{1}{p^{*}}<\frac{1}{p_{0}},
\]
which ensures that the condition $\nu(p_{0}-1) < \alpha$ is satisfied.
\end{remark}
\begin{proof}
For $T>0$, we consider the operator $\Psi$, that is
$$
\Psi(w)=E_{\alpha,1}((-it)^{\alpha}D^{\beta})w_{0}(x) + i^{-\alpha}
\int_{0}^{t}(t-s)^{\alpha-1}E_{\alpha,\alpha}((-i(t-s))^{\alpha}D^{\beta})g(w(s,x)) \, ds, \text{ }t\in[0,T].
$$
By using Lemma \ref{Mittag-Leffer operator estimate}, we obtain that
\begin{align*}
\left\|E_{\alpha,1}((-it)^{\alpha}D^{\beta})w_{0}(\cdot)\right\|_{Y_{p_{0}}^{\nu}}=\sup_{t\in(0,T)}t^{\nu}
\left\|E_{\alpha,1}((-it)^{\alpha}D^{\beta})w_{0}(\cdot)\right\|_{L^{p_{0}}}\lesssim \left\|w_{0}\right\|_{L^{p_{0}'}},
\end{align*}
and
\begin{align*}
&\left\|\int_{0}^{t}(t-s)^{\alpha-1}E_{\alpha,\alpha}((-i(t-s))^{\alpha}D^{\beta})g(w(s,x)) \, ds\right\|_{Y_{p_{0}}^{\nu}}\\
&\lesssim \sup_{t\in(0,T)}t^{\nu}\int_{0}^{t}(t-s)^{\alpha-1-\frac{\alpha d}{\beta}(1-\frac{2}{p_{0}})}\left\|w(s)\right\|_{L^{p_{0}}}^{p_{0}-1}ds\\
&\leq CT^{\alpha-\nu(p_{0}-1)}B(\alpha-\nu,\nu(p_{0}-1))\left\|w\right\|_{Y_{p_{0}}^{\nu}}^{p_{0}-1}.
\end{align*}
Since $w_{0}\in L^{p_{0}'}(\mathbb{R}^{d})$, we can choose constant $M>0$ such that $\left\|w_{0}\right\|_{p_{0}'}\leq \frac{M}{2C}$. Let us consider the closed ball $B_{M}$ in $Y_{p_{0}}^{\nu}$, that is
$$
B_{M}=\left\{w\in Y_{p_{0}}^{\nu}:\quad \left\|w\right\|_{Y_{p_{0}}^{\nu}}\leq M\right\}.
$$
we choose the constant $T>0$ such that
$$
2CT^{\alpha-\nu(p_{0}-1)}B(\alpha-\nu,\nu(p_{0}-1))M^{p_{0}-2}<1,
$$
hence we can get that for $w\in B_{M}$,
\begin{align*}
\left\|\Psi(w)\right\|_{Y_{p_{0}}^{\nu}}\leq C\frac{M}{2C}+CT^{\alpha-\nu(p_{0}-1)}B(\alpha-\nu,\nu(p_{0}-1))M^{p_{0}-1}\leq \frac{M}{2}+\frac{M}{2}=M,
\end{align*}
and for any $w,v\in Y_{p_{0}}^{\nu}$,
\begin{align*}
\left\|\Psi(w)-\Psi(v)\right\|_{Y_{p_{0}}^{\nu}}&\leq 2CT^{\alpha-\nu(p_{0}-1)}B(\alpha-\nu,\nu(p_{0}-1))M^{p-2}\left\|w-v\right\|_{Y_{p_{0}}^{\nu}}\\
&<\left\|w-v\right\|_{Y_{p_{0}}^{\nu}},
\end{align*}
The Theorem \ref{local existence} is now completed by the Banach fixed point theorem.
\end{proof}
\subsubsection{The case $\beta<2$}
When $\beta <2$, Lemma \ref{Mittag-Leffer operator estimate} is no longer applicable, and we cannot establish the well-posedness using the methods in Theorem \ref{global existence} and Theorem \ref{local existence}. Inspired by \cite{Kenig,Grande}, we establish the following local well-posedness result.

\begin{theorem}\label{local-well-posedness}
Let $d=1$, $\beta > \frac{\delta+1}{2}$, $1<\delta < \frac{3}{2}$, $\varsigma=\frac{1}{2} - \frac{1}{2(p-2)}$, and $s\geq\varsigma$, where $p \geq 4$ is an even integer. If the initial value $w_{0} \in H^{s}(\mathbb{R})$, then there exists $T = T(\left\|w_{0}\right\|_{H^{s}(\mathbb{R})})$ such that the E.q.\eqref{Eq:F-T-S-E} has a unique mild solution $w \in C(0, T; H^{s}(\mathbb{R}))$ satisfying
\begin{align}
\label{regularity1} \left\|\langle D \rangle^{\sigma} w\right\|_{L^{\infty}_{x}L^{2}_{T}} &< \infty,\\
\label{regularity2} \left\|\langle D \rangle^{s} w\right\|_{L^{\infty}_{T}L^{2}_{x}} &< \infty, \\
\label{regularity3} \left\|w\right\|_{L^{2(p-2)}_{x}L^{\infty}_{T}} &< \infty, \\
\label{regularity4} \left\|\langle D \rangle^{\theta(s + \delta - \beta)} w\right\|_{L^{\frac{2(p-2)}{1-\theta}}_{x}L^{\frac{2}{\theta}}_{T}} &< \infty,\\
\label{regularity5} \left\|\langle D \rangle^{(1-\theta)(s + \delta - \beta)} w\right\|_{L^{\frac{2(p-2)}{\theta}}_{x}L^{\frac{2}{1-\theta}}_{T}} &< \infty,
\end{align}
where $\sigma=s + (3\delta - 1)/4- \beta/2$, $\theta\in\big(\gamma/(\varsigma+\gamma),\varsigma/(\varsigma+\gamma)\big)$.
\end{theorem}
\begin{remark}
We can choose $\frac{1+\sqrt{1+12\beta}}{6} < \alpha < 1$ such that
$\delta + 1 <2\beta<3\alpha$.
Therefore, the condition $\frac{\delta + 1}{2} < \beta$ and $\delta < \frac{3}{2}$ is reasonable.
\end{remark}
\begin{proof}
Define the operator $\Psi(w)$ as in \eqref{mild solution}. Similar to Kenig \cite[\emph{Commun. Pure Appl. Math.}]{Kenig}, without loss of generality, we only consider the case $s \in [\varsigma, \frac{3}{4})$. Once this result is established, the proof for $s \geq \frac{3}{4}$ becomes simpler, as in our estimates, the highest derivatives always appear linearly. The verification that $\Psi(w) \in C(0, T; H^{s}(\mathbb{R}))$ can be found in \cite{Grande}. Furthermore, based on Lemma \ref{find Mild solution}, up to a constant, we have that
\begin{align*}
\Psi(w)&\sim \left(e^{-itD^{\delta}} + \int_{0}^{\infty}
\frac{e^{-rtD^{\delta}}r^{\alpha-1}}
{i^{\alpha}r^{2\alpha} - 2r^{\alpha}\cos(\alpha\pi) + (-i)^{\alpha}} \, dr\right) w_{0}(x)\\\notag
&\quad + i^{-\alpha} \int_{0}^{t} \left(D^{\delta-\beta}
e^{-i(t-s)D^{\delta}} + \int_{0}^{\infty} \frac{D^{\delta-\beta}
e^{-r(t-s)D^{\delta}} r^\alpha}{i^{\alpha}r^{2\alpha} - 2r^\alpha \cos(\alpha\pi) + (-i)^{\alpha}} \, dr \right) g(w(s,x)) \, ds.
\end{align*}
Consider the following Banach space:
\[
\Theta_{T} = \left\{ w \in C\left(0, T; H^{s}(\mathbb{R})\right) : \lambda(w) = \max_{j=1,2,3,4,5} \lambda_{j}(w) < \infty \right\},
\]
where $\lambda_{1}(w)$, $\lambda_{2}(w)$, $\lambda_{3}(w)$, $\lambda_{4}(w)$, and $\lambda_{5}(w)$ are defined as in \eqref{regularity1}-\eqref{regularity5}.

Next, we sequentially estimate $\lambda_{j}(\Psi(w))$ for $j = 1, 2, 3, 4, 5$.

We choose $\gamma' = (5\delta - 2\beta - 3)/8$. Noting that $\gamma' - \delta + \beta < (\delta - 1)/2$ and setting $\nu' = \gamma' - \delta + \beta$, we have $0<\sigma-\gamma<\sigma-\gamma'<s$ and $0<\sigma-\nu<\sigma-\nu'<s+\delta-\beta$.

By \eqref{smoothing 1}, \eqref{smoothing 9}, \eqref{smoothing 10}, \eqref{smoothing 11}, \eqref{smoothing 12}, \eqref{smoothing 13}, \eqref{smoothing 14}, and \eqref{smoothing 15}, we obtain
\begin{align*}
\lambda_{1}\left(\Psi(w)\right) &\lesssim \left(1 + T^{\frac{1}{2}} + T^{\frac{\gamma - \gamma'}{\delta}}\right)\left\|w_{0}\right\|_{L^{2}_{x}}+\left\|D^{\sigma-\gamma}w_{0}\right\|_{L^{2}_{x}}+T^{\frac{\gamma-\gamma'}{\delta}}
\left\|D^{\sigma-\gamma'}w_{0}\right\|_{L^{2}_{x}}\\
&\quad+ \left(T^{\frac{1}{2}} + T + T^{\frac{\nu}{\delta} + \frac{1}{2}}+T^{\frac{\nu-\nu'}{\delta} + \frac{1}{2}}\right) \left\|(|w|^{p-2} w)\right\|_{L^{2}_{T,x}}+T^{\frac{1}{2}}\big\|D^{\sigma-\nu}(|w|^{p-2}w)\big\|_{L^{2}_{T,x}}\\
&\quad+T^{\frac{\nu-\nu'}{\delta} + \frac{1}{2}}\big\|D^{\sigma-\nu}(|w|^{p-2}w)\big\|_{L^{2}_{T,x}}\\
&\lesssim \left(1 + T^{\frac{1}{2}} + T^{\frac{\gamma - \gamma'}{\delta}}\right) \left\|\langle D \rangle^{s} w_{0}\right\|_{L^{2}_{x}} + \left(T^{\frac{1}{2}} + T + T^{\frac{\nu}{\delta} + \frac{1}{2}}+T^{\frac{\nu-\nu'}{\delta}
+ \frac{1}{2}}\right) \left\|\langle D \rangle^{s + \delta - \beta} (|w|^{p-2} w)\right\|_{L^{2}_{T,x}}.
\end{align*}

By \eqref{smooth 1}, \eqref{smooth 2}, \eqref{smooth 3}, and \eqref{smooth 4}, similar the estimate of $\lambda_{1}\left(\Psi(w)\right)$, we have
\begin{align*}
\lambda_{2}\left(\Psi(w)\right) &\lesssim \left\|\langle D \rangle^{s} w_{0}\right\|_{L^{2}_{x}} + \left(T^{\frac{1}{2}} + T^{\frac{3\alpha-2}{4}}\right) \left\|\langle D \rangle^{s + \delta - \beta} (|w|^{p-2} w)\right\|_{L^{2}_{T,x}}.
\end{align*}

By \eqref{smoothing 16}, \eqref{smoothing 17}, \eqref{smoothing 18}, and \eqref{smoothing 19}, we obtain
\begin{align*}
\lambda_{3}\left(\Psi(w)\right) &\lesssim \left\|\langle D \rangle^{s} w_{0}\right\|_{L^{2}_{x}} + \left(T^{\frac{1}{2}} + T^{\alpha - \frac{1}{2}}\right) \left\|\langle D \rangle^{s + \delta - \beta} (|w|^{p-2} w)\right\|_{L^{2}_{T,x}}.
\end{align*}

By \eqref{smoothing 20}, \eqref{smoothing 21}, \eqref{smoothing 22}, and \eqref{smoothing 23}, we have
\begin{align*}
&\lambda_{4}\left(\Psi(w)\right)\\&\lesssim \left(1 + T^{\frac{(\gamma - \gamma')\theta}{\delta}}\right) \left\|\langle D \rangle^{s - \theta \nu'} w_{0}\right\|_{L^{2}_{x}} + \left(T^{\frac{1}{2}} + T^{\frac{1-2(1-\alpha)(1-\theta)}{2}+\frac{(\nu-\nu')\theta}{\delta}}\right) \left\|\langle D \rangle^{s + \delta - \beta - \theta \nu'} (|w|^{p-2} w)\right\|_{L^{2}_{T,x}} \\
&\lesssim \left(1 + T^{\frac{(\gamma - \gamma')\theta}{\delta}}\right) \left\|\langle D \rangle^{s} w_{0}\right\|_{L^{2}_{x}} + \left(T^{\frac{1}{2}} + T^{\frac{1-2(1-\alpha)(1-\theta)}{2}+\frac{(\nu-\nu')\theta}{\delta}}\right) \left\|\langle D \rangle^{s + \delta - \beta} (|w|^{p-2} w)\right\|_{L^{2}_{T,x}}.
\end{align*}

Similarly,
\begin{align*}
&\lambda_{5}\left(\Psi(w)\right)\\
 &\lesssim\left(1 + T^{\frac{(\gamma - \gamma')(1-\theta)}{\delta}}\right) \left\|\langle D \rangle^{s} w_{0}\right\|_{L^{2}_{x}} + \left(T^{\frac{1}{2}} + T^{\frac{1-2(1-\alpha)\theta}{2}+\frac{(\nu-\nu')(1-\theta)}{\delta}}\right) \left\|\langle D \rangle^{s + \delta - \beta} (|w|^{p-2} w)\right\|_{L^{2}_{T,x}}.
\end{align*}

Noting that $s + \delta - \beta \in (0, 1)$ and $s + \delta - \beta<\sigma$, by Proposition \ref{Leibnitz rule}, we obtain
\begin{align*}
\left\|\langle D \rangle^{s + \delta - \beta} (|w|^{p-2} w)\right\|_{L^{2}_{T,x}} &\lesssim \left\|(\langle D \rangle^{s + \delta - \beta} (|w|^{p-2})) w\right\|_{L^{2}_{T,x}} + \left\|(\langle D \rangle^{s + \delta - \beta} w) |w|^{p-2}\right\|_{L^{2}_{T,x}} \\
&\quad + \left\|\langle D \rangle^{\theta (s + \delta - \beta)} w\right\|_{L^{\frac{2(p-2)}{1-\theta}}_{x} L^{\frac{2}{\theta}}_{T}} \left\|\langle D \rangle^{(1 - \theta) (s + \delta - \beta)} |w|^{p-2}\right\|_{L^{\eta}_{x} L^{\frac{2}{1-\theta}}_{T}},
\end{align*}
where $\eta = \frac{p + \theta - 3}{2(p - 2)}$. Thus, we have
\begin{align*}
\left\|\langle D \rangle^{(1 - \theta) (s + \delta - \beta)} |w|^{p-2}\right\|_{L^{\eta}_{x} L^{\frac{2}{1-\theta}}_{T}} &\lesssim \left\||w|^{p-3}\right\|_{L^{\frac{2(p-2)}{p-3}}_{x} L^{\infty}_{T}} \left\|\langle D \rangle^{(1 - \theta) (s + \delta - \beta)} w\right\|_{L^{\frac{2(p-2)}{\theta}}_{x} L^{\frac{2}{1-\theta}}_{T}} \\
&\lesssim \lambda_{3}(w)^{p-3} \lambda_{5}(w),
\end{align*}
\begin{align*}
\left\|(\langle D \rangle^{s + \delta - \beta} w) |w|^{p-2}\right\|_{L^{2}_{T,x}} &\lesssim \left\|\langle D \rangle^{s + \delta - \beta} w\right\|_{L^{\infty}_{x} L^{2}_{T}} \left\|w\right\|_{L^{2(p-2)}_{x} L^{\infty}_{T}}^{p-2} \\
&\lesssim \lambda_{1}(w) \lambda_{3}(w)^{p-2},
\end{align*}
and
\begin{align*}
\left\|(\langle D \rangle^{s + \delta - \beta} (|w|^{p-2})) w\right\|_{L^{2}_{T,x}} &\lesssim \left\|w\right\|_{L^{2(p-2)}_{x} L^{\infty}_{T}} \left\|\langle D \rangle^{s + \delta - \beta} |w|^{p-2}\right\|_{L^{\frac{2(p-2)}{p-3}}_{x} L^{2}_{T}} \\
&\lesssim \left\|w\right\|_{L^{2(p-2)}_{x} L^{\infty}_{T}}^{p-3} \left\|\langle D \rangle^{\theta (s + \delta - \beta)} w\right\|_{L^{\frac{2(p-2)}{1-\theta}}_{x} L^{\frac{2}{\theta}}_{T}} \left\|\langle D \rangle^{(1 - \theta) (s + \delta - \beta)} w\right\|_{L^{\frac{2(p-2)}{\theta}}_{x} L^{\frac{2}{1-\theta}}_{T}} \\
&\quad + \left\|w\right\|_{L^{2(p-2)}_{x} L^{\infty}_{T}}^{p-2} \left\|\langle D \rangle^{s + \delta - \beta} w\right\|_{L^{\infty}_{x} L^{2}_{T}} \\
&\lesssim \lambda_{3}(w)^{p-3} \lambda_{4}(w) \lambda_{5}(w) + \lambda_{3}(w)^{p-2} \lambda_{1}(w),
\end{align*}
where we have used Remark \ref{A.S.I.E}. Therefore, summarizing the above estimates, we have
\begin{align*}
\left\|\langle D \rangle^{s + \delta - \beta} (|w|^{p-2} w)\right\|_{L^{2}_{T,x}} &\lesssim \lambda_{3}(w)^{p-3} \lambda_{4}(w) \lambda_{5}(w) + \lambda_{3}(w)^{p-2} \lambda_{1}(w) + \lambda_{1}(w) \lambda_{3}(w)^{p-2} \\
&\quad + \lambda_{3}(w)^{p-3} \lambda_{4}(w) \lambda_{5}(w) \\
&\lesssim \lambda(w)^{p-1}.
\end{align*}

Noting that $w_{0} \in H^{s}(\mathbb{R})$, there exists $R > 0$ such that $\left\|w_{0}\right\|_{H^{s}(\mathbb{R})} \leq R/(2C)$. We can choose $0 < T < 1$ such that
\[
2C \big(T + T^{\frac{1}{2}} + T^{\frac{3\alpha-2}{4}}+ T^{\alpha - \frac{1}{2}} + T^{\frac{1}{2} + \frac{\nu}{\delta}} + T^{\frac{1}{2} + \frac{\nu-\nu'}{\delta}}+ T^{\frac{1-2(1-\alpha)(1-\theta)}{2}+\frac{(\nu-\nu')\theta}{\delta}} + T^{\frac{1-2(1-\alpha)\theta}{2}+\frac{(\nu-\nu')(1-\theta)}{\delta}}\big) R^{p-2} < 1.
\]

Considering the closed ball $B(R) \subset \Theta_{T}$, we have
\begin{align*}
&\lambda\left(\Psi(w)\right)\\ &\leq C \frac{R}{2C}+\\
&\quad +  C \left(T + T^{\frac{1}{2}}
+ T^{\frac{3\alpha-2}{4}}+ T^{\alpha - \frac{1}{2}} + T^{\frac{1}{2} + \frac{\nu-\nu'}{\delta}} + T^{\frac{1-2(1-\alpha)(1-\theta)}{2}+\frac{(\nu-\nu')\theta}{\delta}} + T^{\frac{1-2(1-\alpha)\theta}{2}+\frac{(\nu-\nu')(1-\theta)}{\delta}}\right) R^{p-2} \cdot R \\
&\leq \frac{R}{2} + \frac{R}{2} = R.
\end{align*}

Furthermore, for any $v, w \in B(R)$, similar the above process, we obtain
\begin{align*}
&\lambda\left(\Psi(v) - \Psi(w)\right) \\&\leq 2C \left(T + T^{\frac{1}{2}}
+ T^{\frac{3\alpha-2}{4}}+ T^{\alpha - \frac{1}{2}} + T^{\frac{1}{2} + \frac{\nu-\nu'}{\delta}} + T^{\frac{1-2(1-\alpha)(1-\theta)}{2}+\frac{(\nu-\nu')\theta}{\delta}} + T^{\frac{1-2(1-\alpha)\theta}{2}+\frac{(\nu-\nu')(1-\theta)}{\delta}}\right)\\
&\quad\cdot \left(\lambda(v)^{p-2} + \lambda(w)^{p-2}\right) \lambda(v - w) \\
&< \lambda(v - w).
\end{align*}

Therefore, by the Banach fixed-point theorem, equation \eqref{Eq:F-T-S-E} has a unique mild solution $w \in C(0, T; H^{s}(\mathbb{R}))$ that satisfies \eqref{regularity1}-\eqref{regularity5}. The proof is completed now.
\end{proof}

\subsection{Asymptotic behavior and Self-similar solutions}
In this section, we construct the asymptotic behavior and self-similar solutions about E.q. \eqref{Eq:F-T-S-E}.
\begin{theorem}\label{asymptotic behavior}
Let $v$ and $w$ be the global solutions of Equation \eqref{Eq:F-T-S-E} as determined by Theorem \ref{global existence}, with corresponding initial values $v_{0}$ and $w_{0}$, respectively. Then,
\begin{align*}
\lim_{t\rightarrow\infty}t^{\kappa}\left\|w(t)-v(t)\right\|_{L^{p_{0}}}=0 \quad \text{if and only if} \quad \lim_{t\rightarrow\infty}t^{\kappa}\left\|E_{\alpha,1}((-it)^{\alpha}D^{\beta})(w_{0}-v_{0})\right\|_{L^{p_{0}}}=0.
\end{align*}
\end{theorem}
\begin{proof}
On the hand, if $\lim_{t\rightarrow\infty}t^{\kappa}\left\|E_{\alpha,1}((-it)^{\alpha}D^{\beta})(w_{0}-v_{0})\right\|_{L^{p_{0}}}=0$, for any $0<\varrho<1$, we can get that
\begin{align*}
&t^{\kappa}\left\|w(t)-v(t)\right\|_{L^{p_{0}}}\\
&\lesssim t^{\kappa}\left\|E_{\alpha,1}((-it)^{\alpha}D^{\beta})(w_{0}-v_{0})\right\|_{L^{p_{0}}}\\
&\quad +
t^{\kappa}\int_{0}^{t}(t-s)^{\alpha-1}\left\|E_{\alpha,\alpha}((-it)^{\alpha}D^{\beta})
\left(\left(|w(s)|^{p_{0}-2}+|v(s)|^{p_{0}-2}\right)(w(s)-v(s))\right)\right\|_{L^{p_{0}}}ds\\
&\lesssim t^{\kappa}\left\|E_{\alpha,1}((-it)^{\alpha}D^{\beta})(w_{0}-v_{0})\right\|_{L^{p_{0}}}\\
&\quad +t^{\kappa}\left(\int_{0}^{\varrho t}+\int_{\varrho t}^{t}\right)(t-s)^{\alpha-1-\frac{\alpha d}{\beta}(1-\frac{2}{p_{0}})}\left\|
\left(|w(s)|^{p_{0}-2}+|v(s)|^{p_{0}-2}\right)(w(s)-v(s))\right\|_{L^{p_{0}'}}ds\\
&\lesssim t^{\kappa}\left\|E_{\alpha,1}((-it)^{\alpha}D^{\beta})(w_{0}-v_{0})\right\|_{L^{p_{0}}}\\
&\quad +t^{\kappa}\varepsilon^{p_{0}-2}\left(\int_{0}^{\varrho t}+\int_{\varrho t}^{t}\right)(t-s)^{\alpha-1-\frac{\alpha d}{\beta}(1-\frac{2}{p_{0}})}s^{-\kappa(p_{0}-1)}\left(s^{\kappa}\left\|w(s)-v(s)\right\|_{L^{p_{0}}}\right)ds\\
&\lesssim t^{\kappa}\left\|E_{\alpha,1}((-it)^{\alpha}D^{\beta})(w_{0}-v_{0})\right\|_{L^{p_{0}}}\\
&\quad +t^{\kappa}\varepsilon^{p_{0}-2}\left(\int_{0}^{\varrho}+\int_{\varrho}^{1}\right)(1-s)^{\alpha-1-\frac{\alpha d}{\beta}(1-\frac{2}{p_{0}})}s^{-\kappa(p_{0}-1)}\left((ts)^{\kappa}\left\|w(ts)-v(ts)\right\|_{L^{p_{0}}}\right)ds,
\end{align*}
hence we have that
\begin{align*}
\overline{\lim_{t\rightarrow\infty}}t^{\kappa}\left\|w(t)-v(t)\right\|_{L^{p_{0}}}\leq C\varepsilon^{p_{0}-2}\left(1+H(\varrho)\right)\overline{\lim_{t\rightarrow\infty}}t^{\kappa}\left\|w(t)-v(t)\right\|_{L^{p_{0}}},
\end{align*}
note that the constant $\varepsilon$ is small and $\lim_{\varrho\rightarrow 0^{+}}H(\varrho)=0$, we get that
$$
\lim_{t\rightarrow\infty}t^{\kappa}\left\|w(t)-v(t)\right\|_{L^{p_{0}}}=0.
$$
On the other hand, if $\lim_{t\rightarrow\infty}t^{\kappa}\left\|w(t)-v(t)\right\|_{L^{p_{0}}}=0$, note that
\begin{align*}
&t^{\kappa}\left\|E_{\alpha,1}((-it)^{\alpha}D^{\beta})(w_{0}-v_{0})\right\|_{L^{p_{0}}}\\
&\lesssim t^{\kappa}\left\|w(t)-v(t)\right\|_{L^{p_{0}}}+t^{\kappa}\int_{0}^{t}(t-s)^{\alpha-1}\left\|E_{\alpha,\alpha}((-it)^{\alpha}D^{\beta})
\left(\left(|w(s)|^{p_{0}-2}+|v(s)|^{p_{0}-2}\right)(w(s)-v(s))\right)\right\|_{L^{p_{0}}}ds\\
&\leq C\varepsilon^{p_{0}-2}\left(1+H(\varrho)\right)\overline{\lim_{t\rightarrow \infty}}t^{\kappa}\left\|w(t)-v(t)\right\|_{L^{p_{0}}},
\end{align*}
hence we obtain that
$$
\overline{\lim_{t\rightarrow\infty}}t^{\kappa}\left\|E_{\alpha,1}((-it)^{\alpha}D^{\beta})(w_{0}-v_{0})\right\|_{L^{p_{0}}}=0.
$$
The proof is now completed.
\end{proof}
\begin{theorem}\label{self-similar solution}
 Assume the condition Theorem {\rm\ref{global existence}} is satisfied and the initial value $w_{0}$ is a homogeneous function with homogeneity degree $-\frac{\beta}{p_{0}-2}$ and satisfies $\left\|E_{\alpha,1}((-i)^{\alpha}D^{\beta})w_{0}\right\|_{L^{p_{0}}}$ being sufficiently small. Then, the equation \eqref{Eq:F-T-S-E} admits a self-similar solution in $X_{p_{0}}^{\kappa}$, i.e., for all $\lambda > 0$,
\[
w(x,t) = \lambda^{\frac{\beta}{p-2}} w(\lambda x, \lambda^{\delta} t) \quad \text{a.e. } x \in \mathbb{R}^{d}, t > 0.
\]
\end{theorem}
\begin{proof}
Notice that $w_{0}(\lambda x) = \lambda^{-\frac{\beta}{p_{0}-2}} w_{0}(x)$. We can derive
\[
\widehat{w}_{0}(\lambda \xi) = \lambda^{-d + \frac{\beta}{p_{0}-2}} \widehat{w}_{0}(\xi),
\]
and consequently,
\begin{align}\label{Mitleffer self similar}
E_{\alpha,1}((-it)^{\alpha}D^{\beta})w_{0}(x) &= \int_{\mathbb{R}^{d}} E_{\alpha,1}((-it)^{\alpha} |\xi|^{\beta}) \widehat{w}_{0}(\xi) e^{ix \xi} d\xi \quad (\lambda \zeta = \xi) \notag \\
&= \lambda^{d} \int_{\mathbb{R}^{d}} E_{\alpha,1}((-i)^{\alpha} (\lambda^{\delta} t)^{\alpha} |\zeta|^{\beta}) \widehat{w}_{0}(\lambda \zeta) e^{i \lambda x \cdot \zeta} d\zeta \notag \\
&= \lambda^{\frac{\beta}{p_{0}-2}} E_{\alpha,1}((-i)^{\alpha} (\lambda^{\delta} t)^{\alpha} D^{\beta}) w_{0}(\lambda x).
\end{align}
Thus, for any $t > 0$,
\begin{align*}
t^{\kappa} \left\|E_{\alpha,1}((-it)^{\alpha} D^{\beta}) w_{0}(\cdot)\right\|_{L^{p_{0}}} &= \lambda^{\frac{\beta}{p_{0}-2}} t^{\kappa} \left\|E_{\alpha,1}((-i)^{\alpha} (\lambda^{\delta} t)^{\alpha} D^{\beta}) w_{0}(\lambda \cdot)\right\|_{L^{p_{0}}} \\
&= t^{\kappa} t^{-\frac{\alpha}{p_{0}-2} + \frac{\alpha d}{p \beta}} \left\|E_{\alpha,1}((-i)^{\alpha} D^{\beta}) w_{0}(\cdot)\right\|_{L^{p_{0}}} \\
&= \left\|E_{\alpha,1}((-i)^{\alpha} D^{\beta}) w_{0}(\cdot)\right\|_{L^{p_{0}}}.
\end{align*}
Therefore, when $\left\|E_{\alpha,1}((-i)^{\alpha}D^{\beta})w_{0}(\cdot)\right\|_{L^{p_{0}}}$ is sufficiently small, by Theorem \ref{global existence} and the Banach fixed-point theorem, the equation \eqref{Eq:F-T-S-E} has a unique global solution $w \in X_{p_{0}}^{\kappa}$. On the other hand, the fixed point $w \in X_{p_{0}}^{\kappa}$ must be the limit of the following Picard iteration sequence:
\begin{align*}
\overline{w}_{0}(t,x) &= E_{\alpha,1}((-it)^{\alpha} D^{\beta}) w_{0}(x), \\
\overline{w}_{1}(t,x) &= \overline{w}_{0}(t,x) + \int_{0}^{t} (t-s)^{\alpha-1} E_{\alpha,\alpha}((-i(t-s))^{\alpha} D^{\beta}) |\overline{w}_{0}(s,x)|^{p_{0}-2} \overline{w}_{0}(s,x) ds, \\
\overline{w}_{2}(t,x) &= \overline{w}_{0}(t,x) + \int_{0}^{t} (t-s)^{\alpha-1} E_{\alpha,\alpha}((-i(t-s))^{\alpha} D^{\beta}) |\overline{w}_{1}(s,x)|^{p_{0}-2} \overline{w}_{1}(s,x) ds, \\
&\vdots \\
\overline{w}_{n}(t,x) &= \overline{w}_{0}(t,x) + \int_{0}^{t} (t-s)^{\alpha-1} E_{\alpha,\alpha}((-i(t-s))^{\alpha} D^{\beta}) |\overline{w}_{n-1}(s,x)|^{p_{0}-2} \overline{w}_{n-1}(s,x) ds.
\end{align*}
Clearly, by \eqref{Mitleffer self similar}, $\overline{v}_{0}(t,x)$ satisfies the self-similar property, i.e.,
\[
\overline{w}_{0}(t,x) = \lambda^{\frac{\beta}{p_{0}-2}} \overline{w}_{0}(\lambda^{\delta} t, \lambda x).
\]
Furthermore, we can similarly obtain
\begin{align*}
&\overline{w}_{1}(t,x) \\
&= \overline{w}_{0}(t,x) + \int_{0}^{t} (t-s)^{\alpha-1} E_{\alpha,\alpha}((-i(t-s))^{\alpha} D^{\beta}) |\overline{w}_{0}(s,x)|^{p_{0}-2} \overline{w}_{0}(s,x) ds \\
&= \lambda^{\frac{\beta}{p_{0}-2}} \overline{w}_{0}(\lambda^{\delta} t, \lambda x) + \lambda^{\frac{\beta(p_{0}-1)}{p_{0}-2}} \int_{0}^{t} (t-s)^{\alpha-1} E_{\alpha,\alpha}((-i(t-s))^{\alpha} D^{\beta}) |\overline{w}_{0}(\lambda^{\delta} s, \lambda x)|^{p_{0}-2} \overline{w}_{0}(\lambda^{\delta} s, \lambda x) ds \\
&= \lambda^{\frac{\beta}{p_{0}-2}} \overline{w}_{0}(\lambda^{\delta} t, \lambda x) \\
&\quad + \lambda^{\frac{\beta(p_{0}-1)}{p_{0}-2} - \beta} \int_{0}^{t} (\lambda^{\delta} (t-s))^{\alpha-1} E_{\alpha,\alpha}((-i \lambda^{\beta} (t-s))^{\alpha} D^{\beta}) |\overline{w}_{0}(\lambda^{\delta} s, \lambda x)|^{p_{0}-2} \overline{w}_{0}(\lambda^{\delta} s, \lambda x) d(\lambda^{\delta} s) \\
&= \lambda^{\frac{\beta}{p_{0}-2}} \overline{w}_{0}(\lambda^{\delta} t, \lambda x) + \lambda^{\frac{\beta}{p_{0}-2}} \int_{0}^{\lambda^{\delta} t} (\lambda^{\delta} t - s)^{\alpha-1} E_{\alpha,\alpha}((-i (\lambda^{\delta} t - s))^{\alpha} D^{\beta}) |\overline{w}_{0}(s, \lambda x)|^{p_{0}-2} \overline{w}_{0}(s, \lambda x) ds \\
&= \lambda^{\frac{\beta}{p_{0}-2}} \overline{w}_{1}(\lambda^{\delta} t, \lambda x).
\end{align*}
This shows that $\overline{w}_{1}(t,x)$ also satisfies the self-similar property. Similarly, for any $n \in \mathbb{N}$,
\[
\overline{w}_{n}(t,x) = \lambda^{\frac{\beta}{p_{0}-2}} \overline{w}_{n}(\lambda^{\delta} t, \lambda x).
\]
Consequently, in the topology of $X_{p_{0}}^{\kappa}$,
\[
w(t,x) = \lim_{n \to \infty} \overline{w}_{n}(t,x) = \lim_{n \to \infty} \lambda^{\frac{\beta}{p_{0}-2}} \overline{w}_{n}(\lambda^{\delta} t, \lambda x) = \lambda^{\frac{\beta}{p_{0}-2}} w(\lambda^{\delta} t, \lambda x).
\]
The proof is completed now.
\end{proof}
\section{Conclusion}
In this paper, we consider the local and global well-posedness of the \(\mu=\alpha\) type space-time fractional Schr\"{o}dinger equation with power-type nonlinearity, where the linear part is consistent with the equation proposed by Naber. Due to the derivative loss in the solution operator, this equation is fundamentally different from the \(\mu=1\) type space-time fractional Schr\"{o}dinger equation studied in \cite{Su2}. We adopt a strategy similar to the smoothing effects used by Kenig et al. in their study of the Korteweg-de Vries equation, combined with the fractional Leibniz rule established by Kenig and other harmonic analysis methods, to establish the local well-posedness of E.q.\eqref{Eq:F-T-S-E} in $C(0,T;H^{s})$ when \(\beta < 2\) and \(d = 1\). It is worth noting that this method is only effective for \(d = 1\), and due to the limitations imposed by the fractional heat kernel, the global well-posedness for \(\beta < 2\) remains unclear, which is worthy of further discussion. Furthermore, because the H\"{o}rmander multiplier theorem is not applicable due to the derivative loss in the solution operator, we generalize a regularity result for oscillatory integrals established by Ponce in the study of the Korteweg-de Vries equation equation to higher dimensions. By combining real interpolation techniques, we establish $\mathcal{H}^{b}-\mathcal{H}^{a}$ estimates for the operator $D^{\delta-\beta}\exp(-it^{\alpha}D^{\delta})$. Subsequently, combined with harmonic analysis methods such as time-space estimates for the fractional heat kernel, we prove the global and local well-posedness of mild solutions in the space $X_{p_{0}}^{\kappa}$ for space dimensions $d \geq 1$ and $\beta > 2$. On the other hand, adjusting the value of $\alpha$ to $\alpha\in (1,2)$ allows establishing a connection or transition between the Schr\"{o}dinger equation and the classical wave equation, while ensuring that the resulting equation remains within the class of dispersive equations. These equations are characterized by the distinctive phenomenon where different frequency components of waves propagate at distinct speeds. Consequently, in subsequent research, we will investigate the existence of solutions for Equation \eqref{Eq:F-T-S-E} with $\alpha\in(1,2)$.

\noindent{\bf Declaration of competing interest}\\
The authors declare that they have no competing interests.\\
\noindent{\bf Data availability}\\
No data was used for the research described in the article.\\
\noindent{\bf Acknowledgements}\\
This work was supported by National Natural Science Foundation of China (12471172), Fundo para o Desenvolvimento das Ci\^{e}ncias e da Tecnologia of Macau (No. 0092/2022/A) and Hunan Province Doctoral Research Project CX20230633.

\end{document}